\documentclass[10pt]{amsart}
\usepackage{amssymb,amsmath,amsfonts,amscd}

\theoremstyle{plain}
\newtheorem{theorem}{Theorem}[section]

\newtheorem{dfi}[theorem]{Definition}

\newtheorem{cor}[theorem]{Corollary}

\newtheorem{lem}[theorem]{Lemma}

\newtheorem{proposition}[theorem]{Proposition}
\newtheorem{definition}[theorem]{Definition}

\newtheorem{corollary}[theorem]{Corollary}
\newtheorem{remark}[theorem]{Remark}
\newtheorem{lemma}[theorem]{Lemma}

\usepackage{pdfsync}

\newcommand{\rn}[1]{{\mathbb R}^{#1}}
\newcommand{\R}{\mathbb R}

\newcommand{\supp}{\mathrm{supp}\;}

\newcommand{\he}[1]{{\mathbb H}^{#1}}

\newcommand{\cov}[1]{{\bigwedge\nolimits^{#1}{\mfrak h}}}

\newcommand{\covH}[1]{{\bigwedge\nolimits^{#1}{\mfrak h}}}
\newcommand{\vetH}[1]{{\bigwedge\nolimits_{#1}{\mfrak h}}}

\newcommand{\covh}[1]{{\bigwedge\nolimits^{#1}{\mfrak h_1}}}

\newcommand{\scal}[2]{\langle {#1} , {#2}\rangle}

\newcommand{\Scal}[2]{\langle {#1} \vert {#2}\rangle}
\newcommand{\scalp}[3]{\langle {#1} , {#2}\rangle_{#3}}

\newcommand{\ccheck}{{\vphantom i}^{\mathrm v}\!\,}

\newcommand{\mc}{\mathcal }

\newcommand{\mfrak}{\mathfrak}

\newcommand{\WO}[3]{\mathop{W}\limits^\circ{}\!^{{#1},{#2}}
(#3)}

\usepackage[usenames]{color}

\begin{document}

\today

\title[Poincar\'e and Sobolev inequalities for differential forms
] 
{Poincar\'e and Sobolev inequalities for differential forms in Heisenberg groups}

\author[Annalisa Baldi, Bruno Franchi, Pierre Pansu]{
Annalisa Baldi\\
Bruno Franchi\\ Pierre Pansu
}

\keywords{Heisenberg groups, differential forms, Sobolev-Poincar\'e inequalities, contact manifolds, homotopy formula}

\subjclass{58A10,  35R03, 26D15,  43A80,
46E35, 35F35}

\maketitle

\begin{center}
\textbf{Abstract}
\end{center}
Poincar\'e and Sobolev inequalities for differential forms on Heisenberg balls, involving Rumin's differentials, are given. Furthermore, a global homotopy of Rumin's complex which improves differentiability of Rumin forms is provided on any bounded geometry contact manifold.

\section{Introduction}

\subsection{Sobolev and Poincar\'e inequalities for differential forms}

Sobolev inequality in $\R^n$ deals with compactly supported 0-forms, i.e. functions $u$ on $\R^n$, and 1-forms, their differentials $du$. It states that 
\begin{eqnarray*}
\|u\|_q \leq C_{p,q,n}\|du\|_p
\end{eqnarray*} 
whenever
\begin{eqnarray*}
1\leq p,q< +\infty,\quad \frac{1}{p}-\frac{1}{q}=\frac{1}{n}.
\end{eqnarray*}

A local version, for functions supported in the unit ball, holds under the weaker assumption
\begin{eqnarray*}
1\leq p,q< +\infty,\quad \frac{1}{p}-\frac{1}{q}\leq\frac{1}{n}.
\end{eqnarray*}

Poincar\'e's inequality is a variant for functions $u$ defined on but not necessarily compactly supported in the unit ball $B$. It states that there exists a real number $c_u$ such that
\begin{eqnarray*}
\|u-c_u\|_q \leq C_{p,q,n}\|du\|_p.
\end{eqnarray*} 
Alternatively, given a closed 1-form $\omega$ on $B$, there exists a function $u$ on $B$ such that $du=\omega$ on $B$, and such that
\begin{eqnarray*}
\|u\|_q \leq C_{p,q,n}\|\omega\|_p.
\end{eqnarray*} 
This suggests the following generalization for higher degree differential forms.

Let $M$ be a Riemannian manifold. We say that a \emph{strong Poincar\'e inequality} $(p,q)$-Poincar\'e$(k)$ holds on $M$, if 
there exists a positive constant $C=C(M,p,q)$ such that for every closed $k$-form $\omega$ on $M$, belonging to $L^p$, there exists a $k-1$-form $\phi$ such that $d\phi=\omega$ and
\begin{eqnarray*}
\|\phi\|_q \leq C\,\|\omega\|_p.
\end{eqnarray*} 
A \emph{strong Sobolev inequality} $(p,q)$-Sobolev$(k)$ holds on $M$, if for every closed compactly supported $k$-form $\omega$ on $M$, belonging to $L^p$, there exists a compactly supported $k-1$-form $\phi$ such that $d\phi=\omega$ and
\begin{eqnarray*}
\|\phi\|_q \leq C\,\|\omega\|_p.
\end{eqnarray*}
Both statements should be thought of as quantitative versions of the statement that every closed $k$-form is exact.

For Euclidean domains, the validity of Poincar\'e inequality is sensitive to irregularity of boundaries. One way to eliminate such a dependance is to allow a loss on domain. Say an \emph{interior Poincar\'e inequality} $(p,q)$-Poincar\'e$(k)$ holds on $M$ if for every small enough $r>0$ and large enough $\lambda\geq 1$, there exists a constant $C=C(M,p,q,r,\lambda)$ such that for every $x\in M$ and every closed $k$-form $\omega$ on $B(x,\lambda r)$, belonging to $L^p$, there exists a $(k-1)$-form $\phi$ on $B(x,r)$ such that $d\phi=\omega$ on $B(x,r)$ and
\begin{eqnarray*}
\|\phi\|_{L^q(B(x,r))} \leq C\,\|\omega\|_{L^p(B(x,\lambda r))}.
\end{eqnarray*} 
For \emph{interior Sobolev inequalities}, merely add the word compactly supported. Both properties should be thought of as quantitative versions of the statement that, locally, every closed $k$-form is exact.

It turns out that in several situations, the loss on domain is harmless. This is the case for $L^{q,p}$-cohomological applications, see \cite{Pcup}.

\subsection{Contact manifolds}

A contact structure on a manifold $M$ is a smooth distribution of hyperplanes $H$ which is maximally nonintegrable in the following sense: if $\theta$ is a locally defined smooth 1-form such that $H=\mathrm{ker}(\theta)$, then $d\theta$ restricts to a non-degenerate 2-form on $H$. A contact manifold is the data of a smooth manifold $M$ and a contact structure $H$ on $M$. $M$ must be odd-dimensional. Contactomorphisms are contact structure preserving diffeomorphisms between contact manifolds.
The prototype of a contact manifold is the Heisenberg group $\he n$, the simply connected Lie group whose Lie algebra is the central extension $\mathfrak{h}=\mathfrak{h}_1\oplus\mathfrak{h}_2$, $\mathfrak{h}_2=\R=Z(\mathfrak{h})$, with bracket $\mathfrak{h}_1\otimes\mathfrak{h}_1\to\mathfrak{h}_2=\R$ being a non-degenerate skew-symmetric 2-form. The contact structure is obtained by left-translating $\mathfrak{h}_1$. According to Darboux, every contact manifold is locally contactomorphic to $\he n$. The Heisenberg Lie algebra admits a one parameter group of automorphisms $\delta_t$,
\begin{eqnarray*}
\delta_t=t\textrm{ on }\mathfrak{h}_1,\quad \delta_t=t^2 \textrm{ on }\mathfrak{h}_2,
\end{eqnarray*}
which are analogues of Euclidean homotheties. However, differential forms on $\mathfrak{h}$ split into 2 eigenspaces under $\delta_t$, therefore de Rham complex lacks scale invariance under these anisotropic dilations. 

A substitute for de Rham's complex, that recovers scale invariance under $\delta_t$ has been defined by M. Rumin, \cite{rumin_jdg}. It makes sense for arbitrary contact manifolds $(M,H)$. Let $\Omega^\bullet$ denote the space of smooth differential forms on $M$, let $\mathcal{I}^\bullet$ denote the differential ideal generated by 1-forms that vanish on $H$, let $\mathcal{J}^\bullet$ denote its annihilator. Exterior differential $d:\Omega^\bullet\to\Omega^\bullet$ descends to first order differential operators $d_c:\Omega^\bullet/\mathcal{I}^\bullet\to\Omega^\bullet/\mathcal{I}^\bullet$ and $d_c:\mathcal{J}^\bullet\to \mathcal{J}^\bullet$. It turns out that $\Omega^h/\mathcal{I}^h=0$ for $h\geq n+1$ and $\mathcal{J}^h=0$ for $h\leq n$. If $\omega\in \Omega^n/\mathcal{I}^n$, there is a unique lift $\tilde\omega\in\Omega^n$ such that $d\tilde\omega\in\mathcal{J}^{n+1}$. Set $d_c\omega=d\tilde\omega$. This defines a linear second order differential operator $\Omega^n/\mathcal{I}^{n+1}\to\mathcal{J}^{n+1}$ which completes Rumin's complex, which is homotopic to de Rham's complex. The homotopy is a first order differential operator.

Elements of $\Omega^\bullet/\mathcal{I}^\bullet$ and $\mathcal{J}^\bullet$ can be viewed as smooth sections of sub-bundles $\mathcal{E}_0^\bullet$ of $\Lambda^\bullet H^*$ and $\Lambda^\bullet H^*\otimes (TM/H)$ respectively. A Euclidean norm on $H$ determines Euclidean norms on $\Lambda^\bullet H^*$. Locally, a 1-form $\theta$ vanishing on $H$ such that $|d\theta_{|H}|=1$ is uniquely determined up to sign, hence a norm on $TM/H$. The measure on $M$ defined by the locally defined top degree form $\theta\wedge(d\theta)^n$ only depends on the norm on $H$ as well. Whence $L^p$-norms on spaces of sections of bundles $\mathcal{E}_0^\bullet$.

The data of $(M,H)$ equipped with a Euclidean norm defined on sub-bundle $H$ only is called a \emph{sub-Riemannian} contact manifold. Poincar\'e and Sobolev inequalities for differential forms make sense on contact sub-Riemannian manifolds: merely replace $d$ with $d_c$. All left-invariant sub-Riemannian metrics on Heisenberg group are bi-Lipschitz equivalent, hence we may refer to sub-Riemannian Heisenberg group without referring to a specific left-invariant metric. On the other hand, in absence of symmetry assumptions, large scale behaviours of sub-Riemannian contact manifolds are diverse.

\subsection{Results on Poincar\'e and Sobolev inequalities}

In this paper, we prove strong contact Poincar\'e and Sobolev inequalities and interior contact Poincar\'e and Sobolev inequalities in Heisenberg groups, where the word ``contact'' is meant to stress that the exterior differential is replaced by Rumin's $d_c$. The range of parameters differs slightly from the Euclidean case, due to the fact that $d_c$ has order 2 in middle dimension. Let $h\in \{0,\ldots,2n+1\}$. Say that assumption 
$E(h,p,q,n)$ holds if  $1<p\leq q<\infty$ satisfy
\begin{eqnarray*}
\frac{1}{p}-\frac{1}{q}=\begin{cases}
\frac{1}{2n+2}      & \text{ if }h\not=n+1, \\
\frac{2}{2n+2}      & \text{ if }h=n+1.
\end{cases}
\end{eqnarray*}
Say that assumption
$I(h,p,q,n)$ holds if  $1<p\leq q<\infty$ satisfy
\begin{eqnarray*}
\frac{1}{p}-\frac{1}{q}\leq\begin{cases}
\frac{1}{2n+2}      & \text{ if }h\not=n+1, \\
\frac{2}{2n+2}      & \text{ if }h=n+1.
\end{cases}
\end{eqnarray*}

\begin{theorem}\label{strongglobal}
Under assumption $E(h,p,q,n)$, strong $(p,q)$-Poincar\'e and $(p,q)$-Sobolev inequalities hold for $h$-forms on $\he n$. 
\end{theorem}

\begin{theorem}
Under assumption $I(h,p,q,n)$, interior $(p,q)$-Poincar\'e and $(p,q)$-Sobolev inequalities hold for $h$-forms on $\he n$. 
\end{theorem}
Precise formulations of interior Poincar\'e and Sobolev inequalities are given in section \ref{poincare}.

Here is a sample consequence of these results. Combining both theorems with results from \cite{Pcup}, we get
\begin{cor}
Under assumption $E(h,p,q,n)$, the $\ell^{q,p}$-cohomology in degree $h$ of $\he n$ vanishes.
\end{cor}

\subsection{Bounded geometry and smoothing}

Along the way, we construct local smoothing operators for differential forms. They can be combined to yield a global smoothing operator on sub-Riemannian contact manifolds, which has independent interest (see Theorem
\ref{1.5} below).  This operator is bounded on $L^p$ provided the sub-Riemannian metric has bounded geometry in the following sense.

\begin{dfi}\label{contact}
Let $k\geq 2$. Let $B(e,1)$ denote the unit sub-Riemannian ball in $\he n$. We say that a sub-Riemannian contact manifold $(M,H,g)$ has \emph{bounded $C^k$-geometry} is there exist constants $r>0$, $C$ such that, for every $x\in M$, if we denote by $B(x,r)$ the sub-Riemannian ball for $(M,H,g)$ centered at $x$ and of radius $r$,
there exists a contactomorphism (i.e. a diffeomorphism preserving the contact forms) $\phi_x : B(e,1)\to M$
\begin{enumerate}
  \item $B(x,r)\subset\phi_x(B(e,1))$.
  \item $\phi_x$ is $C$-bi-Lipschitz.
  \item Coordinate changes $\phi_x\circ\phi_y^{-1}$ and their first $k$ derivatives with respect to unit left-invariant horizontal vectorfields are bounded by $C$.
\end{enumerate}
\end{dfi}

On sub-Riemannian Heisenberg balls, Sobolev spaces can be defined as follows. Fix an orthonormal basis of left-invariant vector fields $W_i$. Express forms in this frame, and differentiate along these vector-fields only. Let $\ell=0,\ldots,k$. Say that a differential form on unit ball $B$ belongs to $W^{\ell,p}$ if all derivatives up to order $k$ of its components belong to $L^p(B)$. Using $C^k$-bounded charts, this local notion extends to $C^k$-bounded geometry sub-Riemannian contact manifolds $M$, and the global $W^{k,p}$ norm on globally defined differential forms is defined by
\begin{eqnarray*}
\left(\sum_j\|\omega_{|B(x_j,r)}\|_{W^{k,p}(B(x_j,r))}\right)^{1/p},
\end{eqnarray*}
where $x_i$ is an $r$-dense uniformly discrete subset of $M$ (it will be shown in section \ref{function spaces} that this norm does not depend on choices, up to multiplicative constants). By duality, Sobolev spaces with negative $\ell=-k+1,\ldots,-1$ can be defined.

\begin{theorem}\label{1.5}
Let $(M,H,g)$ be a sub-Riemannian contact manifold of bounded $C^k$-geometry. Under assumption $I(h,p,q,n)$, there exist operators $S$ and $T$ on $h$-forms on $M$ which are bounded from $W^{j-1,p}$ to $W^{j,q}$ for all $0\leq j\leq k$, and such that $1=S+d_c T +Td_c$. 
\end{theorem}
Iterating $S$ yields an operator which is bounded from $L^{p}$ to $W^{k,q}$, and still acts trivially on cohomology. For instance, this allows to replace a closed form, up to adding a controlled exact form, with a much more regular differential form.

\subsection{Questions}
Keeping in mind the analogous inequalities in the scalar case, the following questions
naturally arise.
 \begin{itemize} 
\item[1.] Do balls is Heisenberg group satisfy strong $(p,q)$-Poincar\'e and $(p,q)$-Sobolev inequalities? In other words, do Poincar\'e and Sobolev inequalities hold without lack on domain?

\item[2.] Do interior $(p,q)$-Poincar\'e and $(p,q)$-Sobolev inequalities hold for limiting values, i.e. for $p=1$ or $q=\infty$?

\item[3.] How much of these results does extend to more general Carnot groups?
\end{itemize}

\section{Scheme of proof}

\subsection{Global homotopy operators}

The most efficient way to prove a Poincar\'e inequality is to find a homotopy between identity and 0 on the complex of differential forms, i.e. a linear operator $K$ that raises the degree by 1 and satisfies
\begin{eqnarray*}
1=dK+Kd.
\end{eqnarray*}
More generally, we shall deal with homotopies between identity and other operators $P$, i.e. of the form
\begin{eqnarray*}
1-P=dK+Kd.
\end{eqnarray*}

In Euclidean space, the Laplacian provides us with such a homotopy. Write $\Delta=d\delta+\delta d$. Denote by $\Delta^{-1}$ the operator of convolution with the fundamental solution of the Laplacian. Then $\Delta^{-1}$ commutes with $d$ and its adjoint $\delta$, hence $K_e=\delta \Delta^{-1}$ satisfies $1=dK_e+K_ed$ on globally defined $L^p$ differential forms. Furthermore, $K_e$ is bounded $L^p\to W^{1,q}$ provided $\frac{1}{p}-\frac{1}{q}=\frac{1}{n}$. This proves the strong $(p,q)$-Poincar\'e inequality for Euclidean space. Rumin defines a Laplacian $\Delta_c$ by $\Delta_c=d_c\delta_c+\delta_c d_c$ when both $d_c$'s are first order, and by $\Delta_c=(d_c\delta_c)^2+\delta_c d_c$ or $\Delta_c=d_c\delta_c+(\delta_c d_c)^2$ near middle dimension, when one of them has order 2. This leads to a homotopy of the form $K_0=\delta_c \Delta_c^{-1}$ or $K_0=\delta_c d_c\delta_c \Delta_c^{-1}$ depending on degree. Again, $K_0$ is bounded $L^p\to W^{1,q}$ under assumption $E(h,p,q,n)$. This proves the strong contact $(p,q)$-Poincar\'e$(h)$ inequality for Heisenberg group, Theorem \ref{strongglobal}.

\subsection{Local homotopy operators}

We pass to local results. In Euclidean space, Poincar\'e's Lemma asserts that every closed form on a ball is exact. We need a quantitative version of this statement. The standard proof of Poincar\'e's Lemma relies on a homotopy operator which depends on the choice of an origin. Averaging over origins yields a bounded operator $K:L^p\to L^q$, as was observed by Iwaniec and Lutoborski, \cite{IL}. This proves the strong Euclidean $(p,q)$-Poincar\'e$(h)$ inequality for convex Euclidean domains. A support preserving variant $J:L^p\to L^q$ appears in Mitrea-Mitrea-Monniaux, \cite{mitrea_mitrea_monniaux} and this proves the strong Euclidean $(p,q)$-Sobolev inequality for bounded convex Euclidean domains. Incidentally, since, for balls, constants do not depend on the radius of the ball, this reproves the strong Euclidean $(p,q)$-Sobolev inequality for Euclidean space.

In this paper a sub-Riemannian counterpart is obtained using the homotopy of de Rham's and Rumin's complexes. Since this homotopy is a differential operator, a preliminary smoothing operation is needed. This is obtained by localizing (multiplying the kernel with cut-offs) the global homotopy $K_0$ provided by the inverse of Rumin's (modified) Laplacian.

Hence the proof goes as follows (see Section \ref{poincare}):
\begin{enumerate}
  \item Show that the inverse $K_0$ of Rumin's modified Laplacian on all of $\he n$ is given by a homogeneous kernel $k_0$. Deduce bounds $L^p\to W^{1,q}$. Conclude that $K_0$ is an exact homotopy for globally defined $L^p$ forms. 
    \item Split $k_0=k_1+k_2$ where $k_1$ has small support and $k_2$ is smooth. Hence $T=K_1$ is a homotopy on balls (with a loss on domain) of identity to $S=d_c K_2+K_2d_c$ which is smoothing. This provides the required local smoothing operation.
    \item Compose Iwaniec and Lutoborski's averaged Poincar\'e homotopy for the de Rham complex and Rumin's homotopy, and apply the result to smoothed forms. This proves an interior Poincar\'e inequality in Heisenberg group. Replacing Iwaniec and Lutoborski's homotopy with Mitrea-Mitrea-Monniaux's homotopy leads to an interior Sobolev inequality.
\end{enumerate}

\subsection{Global smoothing}

Let $(M,H,g)$ be a bounded $C^k$-geometry sub-Riemannian contact manifold. Pick a uniform covering by equal radius balls. Let $\chi_j$ be a partition of unity subordinate to this covering. Let $\phi_j$ be the corresponding charts from the unit Heisenberg ball. Let $S_j$ and $T_j$ denote the smoothing and homotopy operators transported by $\phi_j$. Set
\begin{eqnarray*}
T=\sum_j T_j \chi_j ,\quad S=\sum_j S_j \chi_j +T_j[\chi_j,d_c].
\end{eqnarray*}
When $d_c$ is first order, the commutator $[\chi_j,d_c]$ is an order 0 differential operator, hence $T_j[\chi_j,d_c]$ gains 1 derivative. When $d_c$ is second order, $[\chi_j,d_c]$ is a first order differential operator. It turns out that precisely in this case, $T_j$ gains 2 derivatives, hence $T_j[\chi_j,d_c]$ gains 1 derivative in this case as well.

This is detailed in section \ref{final}.

\section{Heisenberg groups and Rumin's complex $(E_0^\bullet,d_c)$}
\label{Rumin}

\subsection{Differential forms on Heisenberg group}

 We denote by  $\he n$  the $n$-dimensional Heisenberg
group, identified with $\rn {2n+1}$ through exponential
coordinates. A point $p\in \he n$ is denoted by
$p=(x,y,t)$, with both $x,y\in\rn{n}$
and $t\in\R$.
   If $p$ and
$p'\in \he n$,   the group operation is defined by
\begin{equation*}
p\cdot p'=(x+x', y+y', t+t' + \frac12 \sum_{j=1}^n(x_j y_{j}'- y_{j} x_{j}')).
\end{equation*}
The unit element of $\he n$ is the origin, that will be denote by $e$.

For a general review on Heisenberg groups and their properties, we
refer to \cite{Stein}, \cite{GromovCC} and to \cite{VarSalCou}.
We limit ourselves to fix some notations, following \cite{FSSC_advances}.

The Heisenberg group $\he n$ can be endowed with the homogeneous
norm (Kor\'anyi norm)
\begin{equation}\label{gauge}
\varrho (p)=\big(|p'|^4+p_{2n+1}^2\big)^{1/4},
\end{equation}
and we define the gauge distance (a true distance, see
 \cite{Stein}, p.\,638, that is equivalent to
Carnot--Carath\'eodory distance)
as
\begin{equation}\label{def_distance}
d(p,q):=\varrho ({p^{-1}\cdot q}).
\end{equation}
Finally, set $B_{\rho}(p,r)=\{q \in  \he n; \; d(p,q)< r\}$.

A straightforward computation shows that there exists $c_0>1$ such that
\begin{equation}\label{c0}
c_0^{-2} |p| \le \rho(p) \le |p|^{1/2},
\end{equation}
provided $p$ is close to $e$.
In particular, for $r>0$ small, if we denote by $B_{\mathrm{Euc}}(e,r)$ the Euclidean ball centred ad $e$ of radius $r$,
\begin{equation}\label{balls inclusion}
B_{\mathrm{Euc}}(e,r^2) \subset B_{\rho}(e,r) \subset B_{\mathrm{Euc}}(e, c_0^2 r).
\end{equation}

It is well known that the topological dimension of $\he n$ is $2n+1$,
since as a smooth manifold it coincides with $\R^{2n+1}$, whereas
the Hausdorff dimension of $(\he n,d)$ is $Q:=2n+2$.

    We denote by  $\mfrak h$
 the Lie algebra of the left
invariant vector fields of $\he n$. The standard basis of $\mfrak
h$ is given, for $i=1,\dots,n$,  by
\begin{equation*}
X_i := \partial_{x_i}-\frac12 y_i \partial_{t},\quad Y_i :=
\partial_{y_i}+\frac12 x_i \partial_{t},\quad T :=
\partial_{t}.
\end{equation*}
The only non-trivial commutation  relations are $
[X_{j},Y_{j}] = T $, for $j=1,\dots,n.$ 
The {\it horizontal subspace}  $\mfrak h_1$ is the subspace of
$\mfrak h$ spanned by $X_1,\dots,X_n$ and $Y_1,\dots,Y_n$.
Coherently, from now on, we refer to $X_1,\dots,X_n,Y_1,\dots,Y_n$
(identified with first order differential operators) as to
the {\it horizontal derivatives}. Denoting  by $\mfrak h_2$ the linear span of $T$, the $2$-step
stratification of $\mfrak h$ is expressed by
\begin{equation*}
\mfrak h=\mfrak h_1\oplus \mfrak h_2.
\end{equation*}

\bigskip

The stratification of the Lie algebra $\mfrak h$ induces a family of non-isotropic dilations
$\delta_\lambda$, $\lambda>0$ in $\he n$. The homogeneous dimension  of $\he n$
with respect to $\delta_\lambda$, $\lambda>0$ equals $Q$.

The vector space $ \mfrak h$  can be
endowed with an inner product, indicated by
$\scalp{\cdot}{\cdot}{} $,  making
    $X_1,\dots,X_n$,  $Y_1,\dots,Y_n$ and $ T$ orthonormal.
    
Throughout this paper, we write also
\begin{equation}\label{campi W}
W_i:=X_i, \quad W_{i+n}:= Y_i, \quad W_{2n+1}:= T, \quad \text
{for }i =1, \cdots, n.
\end{equation}

The dual space of $\mfrak h$ is denoted by $\covH 1$.  The  basis of
$\covH 1$,  dual to  the basis $\{X_1,\dots , Y_n,T\}$,  is the family of
covectors $\{dx_1,\dots, dx_{n},dy_1,\dots, dy_n,\theta\}$ where 
$$ \theta
:= dt - \frac12 \sum_{j=1}^n (x_jdy_j-y_jdx_j)$$ is called the {\it contact
form} in $\he n$. 

We indicate as $\scalp{\cdot}{\cdot}{} $ also the
inner product in $\covH 1$  that makes $(dx_1,\dots, dy_{n},\theta  )$ 
an orthonormal basis.

Coherently with the previous notation \eqref{campi W},
we set
\begin{equation*}
\omega_i:=dx_i, \quad \omega_{i+n}:= dy_i, \quad \omega_{2n+1}:= \theta, \quad \text
{for }i =1, \cdots, n.
\end{equation*}

We put
$       \vetH 0 := \covH 0 =\R $
and, for $1\leq k \leq 2n+1$,
\begin{equation*}
\begin{split}
         \covH k& :=\mathrm {span}\{ \omega_{i_1}\wedge\dots \wedge \omega_{i_k}:
1\leq i_1< \dots< i_k\leq 2n+1\}
.
\end{split}
\end{equation*}

The volume $(2n+1)$-form $ \theta_1\wedge\cdots\wedge \theta_{ 2n+1}$
 will be also
written as $dV$.

%

The same construction can be performed starting from the vector
subspace $\mfrak h_1\subset \mfrak h$,
obtaining the {\it horizontal $k$-covectors} 
\begin{equation*}
\begin{split}
         \covh k& :=\mathrm {span}\{ \omega_{i_1}\wedge\dots \wedge \omega_{i_k}:
1\leq i_1< \dots< i_k\leq 2n\}.
\end{split}
\end{equation*}
%

%
%

%
%
%
%
%
%
%

\begin{definition}\label{weight} If $\eta\neq 0$, $\eta\in \covh 1$,  
 we say that $\eta$ has \emph{weight $1$}, and we write
$w(\eta)=1$. If $\eta = \theta$, we say $w(\eta)= 2$.
More generally, if
$\eta\in \covH h$, we say that $\eta$ has \emph {pure weight} $k$ if $\eta$ is
a linear combination of covectors $\omega_{i_1}\wedge\cdots\wedge\omega_{i_h}$
with $w(\omega_{i_1})+\cdots + w(\omega_{ i_h})=k$.

Notice that, if $\eta,\zeta \in \covH h$ and $w(\eta)\neq w(\zeta)$, then
$\scal{\eta}{\zeta}=0$. 
\end{definition}

\subsection{Rumin's complex on Heisenberg groups}

The exterior differential $d$ does not preserve weights. It splits into
\begin{eqnarray*}
d=d_0+d_1+d_2
\end{eqnarray*}
where $d_0$ preserves weight, $d_1$ increases weight by 1 unit and $d_2$ increases weight by 2 units. $d_0$ is a differential operator of order 0; in degree $k$, it vanishes on forms of weight $k$ and if $\beta$ is a $k-1$-form of weight $k-1$, $d_0(\theta\wedge\beta)=d\theta\wedge\beta$. A first attempt in trying to invert $d$ is to invert $d_0$. For this, let us pick a complement $\mathcal{W}$ to $\mathrm{ker}(d_0)$ in $\covH\bullet$ and a complement $\mathcal{V}$ to $\mathrm{Im}(d_0)$ in $\covH\bullet$ containing $\mathcal{W}$. This allows to define $d_0^{-1}$ to be 0 on $\mathcal{V}$ and the inverse of $d_0:\mathcal{W}\to \mathrm{Im}(d_0)$. This defines a left-invariant order 0 operator on smooth forms on $\he n$. Denote by $V$ (resp. $W$) the space of smooth sections of $\mathcal{V}$ (resp. $\mathcal{W}$).

Rumin shows that
\begin{eqnarray*}
r=1-d_0^{-1}d-dd_{0}^{-1}
\end{eqnarray*}
is the projector onto the subspace
\begin{eqnarray*}
E=V\cap d^{-1}V
\end{eqnarray*}
along the subspace
\begin{eqnarray*}
F=W+dW.
\end{eqnarray*}
Hence, in the sequel, it will be denoted by $\Pi_{E}$. The weight-preserving part of $r$,
\begin{eqnarray*}
r_0=1-d_0^{-1}d_0-d_0d_{0}^{-1},
\end{eqnarray*}
has order 0, it is the projector onto $\mathcal{E}_0:=\mathcal{V}\cap\mathrm{ker}(d_0)$ along $\mathcal{W}\oplus\mathrm{Im}(d_0)$. Hence, in the sequel, it will be denoted by $\Pi_{E_0}$, where $E_0$ is the space of smooth sections of $\mathcal{E}_0$. ${\Pi_{E_0}}_{|E}$ and ${\Pi_{E}}_{|E_0}$ are inverses of each other. We use them to conjugate $d_{|E}$ to an operator
\begin{eqnarray*}
d_c=\Pi_{E_0} d\Pi_{E}\Pi_{E_0}
\end{eqnarray*}
on $E_0$. By construction, the complex $(E_0,d_c)$ is isomorphic to $(E,d)$, which is homotopic to the full de Rham complex.

\subsection{Contact manifolds}

We now sketch Rumin's construction of the intrinsic complex for general contact manifolds $(M,H)$. Locally, $H$ is the kernel of a smooth contact 1-form $\theta$. Let $L:\bigwedge^\bullet H^*\to\bigwedge^\bullet H^*$ denote multiplication by $d\theta_{|H}$. 

It is well known that, for every $h\leq n-1$, $L^{n-h}:\bigwedge^h H^*\to\bigwedge^{2n-h} H^*$ is an isomorphism. It follows that $\mathrm{ker}(L^{n-h+1})$ is a complement of $\mathrm{Im}(L)$ in $\bigwedge^h H^*$, if $h\leq n$, and that $\mathrm{Im}(L)=\bigwedge^h H^*$ if $h\geq n+1$. Therefore we set
\begin{eqnarray*}
\mathcal{V}^h=\begin{cases}
\{\alpha\in T^*M\,;\,L^{n-h+1}(\alpha_{|H})=0\} & \text{if }h\leq n, \\
\{\alpha\in T^*M\,;\,\alpha_{|H}=0\}& \text{otherwise}.
\end{cases}
\end{eqnarray*}
Similarly, $\mathrm{Im}(L^{h-n+1})$ is a complement of $\mathrm{ker}(L)$ in $\bigwedge^h H^*$ if $h\geq n$, and $\mathrm{ker}(L)=\{0\}$ in $\bigwedge^h H^*$ if $h\leq n-1$. Therefore we set
\begin{eqnarray*}
\mathcal{W}^h=\begin{cases}
\{\alpha\in T^*M\,;\,\alpha_{|H}=0\} & \text{if }h\leq n-1, \\
\{\alpha\in T^*M\,;\,\alpha\in\theta\wedge \mathrm{Im}(L^{h-n+1})\}& \text{otherwise}.
\end{cases}
\end{eqnarray*}
Changing $\theta$ to an other smooth 1-form $\theta'=f\theta$ with kernel $H$ does not change $\mathcal{V}$ and $\mathcal{W}$. With these choices, spaces of smooth sections $V$ and $W$ depend only on the plane field $H$. We can define subspaces of smooth differential forms $E=V\cap d^{-1}V$ and $F=W+dW$ and the projector $\Pi_E$. Since no extra choices are involved, $E$, $F$ and $\Pi_E$ are invariant under contactomorphisms.

In degrees $h\geq n+1$, $\mathcal{E}_0=\theta\wedge(\bigwedge^h H^* \cap \mathrm{ker}(L))$ is a contact invariant. Since 
\begin{eqnarray*}
(\Pi_{E_0})_{|E}=((\Pi_E)_{|E_0})^{-1},
\end{eqnarray*} 
the operator $d_c=((\Pi_E)_{|E_0})^{-1}\circ d\circ (\Pi_E)_{|E_0}$ is a contact invariant.

In degrees $h\leq n$, the restriction of differential forms to $H$ is an isomorphism of $\mathcal{E}_0$ to $\mathcal{E}'_0:=\bigwedge^h H^* \cap \mathrm{ker}(L^{n-h+1})$. We note that for a differential form $\omega$ such that $\omega_{|H}\in\mathcal{E}'_0$, $\Pi_E(\omega)$ only depends on $\omega_{|H}$. Indeed, $d_0^{-1}\omega=0$. Furthermore, if $\omega=\theta\wedge\beta$, $d_0^{-1}d\omega=d_0^{-1}(d\theta\wedge\beta)=\omega$, hence $\Pi_E(\omega)=\omega-dd_0^{-1}\omega=0$. It follows that $(\Pi_E)_{|E_0}$ can be viewed as defined on the space $E'_0$ of sections of $\mathcal{E}'_0$, which is a contact invariant. Since 
\begin{eqnarray*}
(\Pi_{E_0})_{|E}=((\Pi_E)_{|E_0})^{-1},\quad \textrm{it follows that }(\Pi_{E'_0})_{|E}=((\Pi_E)_{|E'_0})^{-1}
\end{eqnarray*} 
and $d_c$ viewed as an operator on $E'_0$,
\begin{eqnarray*}
((\Pi_E)_{|E'_0})^{-1}\circ d\circ (\Pi_E)_{|E'_0}
\end{eqnarray*}
is a contact invariant. In the sequel, we shall ignore the distinction between $E_0$ and $E'_0$. The connection with the description provided in the introduction is easy.

Alternate contact invariant descriptions of Rumin's complex can be found in \cite{Bernig_2017} and \cite{BEGN}.

By construction,
\begin{itemize}
\item[i)] $d_c^2=0$;
\item[ii)] the complex $\mc E_0:=(E_0^\bullet,d_c)$ is homotopically equivalent to the de Rham complex
$\Omega:= (\Omega^\bullet,d)$. Thus,  if $D\subset \he n$ is an open set, unambiguously we write $H^h(D)$
for the $h$-th cohomology group;
\item[iii)] $d_c: E_0^h\to E_0^{h+1}$ is a homogeneous differential operator in the 
horizontal derivatives
of order 1 if $h\neq n$, whereas $d_c: E_0^n\to E_0^{n+1}$ is an homogeneous differential operator in the 
horizontal derivatives
of order 2.
\end{itemize}

Since the exterior differential $d_c$ on $E_0^h$ can be written in coordinates as a left-invariant homogeneous differential operator in the horizontal variables of order 1 if $h\neq n$ and of order
2 if $h=n$, the proof of the following Leibniz' formula is easy.

\begin{lemma}\label{leibniz} If $\zeta$ is a smooth real function, then
\begin{itemize}
\item if $h\neq n$, then on $E_0^h$ we have:
$$
[d_c,\zeta] = P_0^h(W\zeta),
$$
where $P_0^h(W\zeta): E_0^h \to E_0^{h+1}$ is a homogeneous differential operator of degree zero with coefficients depending
only on the horizontal derivatives of $\zeta$;
\item if $h= n$, then on $E_0^n$ we have
$$
[d_c,\zeta] = P_1^n(W\zeta) + P_0^n(W^2\zeta) ,
$$
where $P_1^n(W\zeta):E_0^n \to E_0^{n+1}$ is a homogeneous differential operator of degree 1 with coefficients depending
only on the horizontal derivatives of $\zeta$, and where $P_0^h(W^2\zeta): E_0^n \to E_0^{n+1}$ is a homogeneous differential operator in
the horizontal derivatives of degree 0
 with coefficients depending
only on second order horizontal derivatives of $\zeta$.
\end{itemize}

\end{lemma}

\section{Kernels}\label{kernels}

  If $f$ is a real function defined in $\he n$, we denote
    by $\ccheck f$ the function defined by $\ccheck f(p):=
    f(p^{-1})$, and, if $T\in\mc D'(\he n)$, then $\ccheck T$
    is the distribution defined by $\Scal{\ccheck T}{\phi}
    :=\Scal{T}{\ccheck\phi}$ for any test function $\phi$.
    
    Following e.g. \cite{folland_stein}, we can define a group
convolution in $\he n$: if, for instance, $f\in\mc D(\he n)$ and
$g\in L^1_{\mathrm{loc}}(\he n)$, we set
\begin{equation}\label{group convolution}
f\ast g(p):=\int f(q)g(q^{-1}\cdot p)\,dq\quad\mbox{for $q\in \he n$}.
\end{equation}
We remind that, if (say) $g$ is a smooth function and $P$
is a left invariant differential operator, then
$$
P(f\ast g)= f\ast Pg.
$$
We remind also that the convolution is again well defined
when $f,g\in\mc D'(\he n)$, provided at least one of them
has compact support. In this case the following identities
hold
\begin{equation}\label{convolutions var}
\Scal{f\ast g}{\phi} = \Scal{g}{\ccheck f\ast\phi}
\quad
\mbox{and}
\quad
\Scal{f\ast g}{\phi} = \Scal{f}{\phi\ast \ccheck g}
\end{equation}
 for any test function $\phi$.
 
 As in \cite{folland_stein},
we also adopt the following multi-index notation for higher-order derivatives. If $I =
(i_1,\dots,i_{2n+1})$ is a multi--index, we set  
$W^I=W_1^{i_1}\cdots
W_{2n}^{i_{2n}}\;T^{i_{2n+1}}$. 
By the Poincar\'e--Birkhoff--Witt theorem, the differential operators $W^I$ form a basis for the algebra of left invariant
differential operators in $\he n$. 
Furthermore, we set 
$|I|:=i_1+\cdots +i_{2n}+i_{2n+1}$ the order of the differential operator
$W^I$, and   $d(I):=i_1+\cdots +i_{2n}+2i_{2n+1}$ its degree of homogeneity
with respect to group dilations.

 Suppose now $f\in\mc E'(\he n)$ and $g\in\mc D'(\he n)$. Then,
 if $\psi\in\mathcal D(\he n)$, we have
 \begin{equation}\label{convolution by parts}
 \begin{split}
\Scal{(W^If)\ast g}{\psi}&=
 \Scal{W^If}{\psi\ast \ccheck g} =
  (-1)^{|I|}  \Scal{f}{\psi\ast (W^I \,\ccheck g)} \\
&=
 (-1)^{|I|} \Scal{f\ast \ccheck W^I\,\ccheck g}{\psi}.
\end{split}
\end{equation}

\medskip

Following \cite{folland}, we remind now the notion of {\it kernel of type $\mu$}.

\begin{definition}\label{type} A kernel of type $\mu$ is a 
homogeneous distribution of degree $\mu-Q$
(with respect to group dilations $\delta_r$),
that is smooth outside of the origin.

The convolution operator with a kernel of type $\mu$
is still called an operator of type $\mu$.
\end{definition}

\begin{proposition}\label{kernel}
Let $K\in\mc D'(\he n)$ be a kernel of type $\mu$.
\begin{itemize}
\item[i)] $\ccheck K$ is again a kernel of type $\mu$;
\item[ii)] $WK$ and $KW $ are associated with  kernels of type $\mu-1$ for
any horizontal derivative $W$;
\item[iii)]  If $\mu>0$, then $K\in L^1_{\mathrm{loc}}(\he n)$.
\end{itemize}
\end{proposition}

\begin{theorem} \label{hls folland} Suppose $0<\alpha<Q$, and let
$K$ be a kernel of type $\alpha$. Then
\begin{itemize}
\item[i)] if $1<p<Q/\alpha$, and $1/q:= 1/p-\alpha/Q$, then
$$
\| u\ast K\|_{L^q(\he n)} \le C \| u\|_{L^p(\he n)}
$$
for all $u\in L^p(\he n)$.
\item[ii)] If $p\ge Q/\alpha$ and $B, B' \subset \he n$ are fixed balls, then
for any $q\ge p$
$$
\| u\ast K\|_{L^q(B')} \le C \| u\|_{L^p(\he n)}
$$
for all $u\in L^p(\he n)$ with $\supp u\subset B$.
\item[iii)] If $K$ is a kernel of type 0 and $1<p<\infty$, then
 $$
\| u\ast K\|_{L^p(\he n)} \le C \| u\|_{L^p(\he n)}.
$$
\end{itemize}

\end{theorem}

\begin{proof}
For statements i) and iii), we refer to \cite{folland}, Propositions 1.11 and 1.9. As for ii),
if $p\ge Q/\alpha$, we choose $1<\tilde p<Q/\alpha$ such that $1/{\tilde p} \le 1/q + \alpha/Q$. If we
set $1/{\tilde q}:= 1/{\tilde p} -\alpha/Q<1/q$,   then
\begin{equation*}\begin{split}
\| u\ast & K\|_{L^{ q}(B')}   \le C_{B'}\| u\ast K\|_{L^{ \tilde q}(B')} \le C_{B'}
\| u\ast K\|_{L^{\tilde q}(\he n)} 
\\&
\le C'(B') \| u\|_{L^{\tilde p}(\he n)} \le C'(B,B') \| u\|_{L^{ p}(B)}.
\end{split}\end{equation*}
\end{proof}

\begin{lemma}\label{truncation} Suppose $0< \alpha<Q$.
If $K$ is a kernel of type $\alpha$
and $\psi \in \mc D(\he n)$, $\psi\equiv 1$ in a neighborhood of the origin, then
the statements i) and ii) of Proposition \ref{hls folland} still
hold if we replace $K$ by $(1-\psi )K$. 

Analogously, if $K$ is a kernel of type 0 and $\psi \in \mc D(\he n)$,
then statement iii) of Proposition \ref{hls folland} still
hold if we replace $K$ by $(\psi-1) K$.
\end{lemma}

\begin{proof} As in \cite{folland}, Proposition 1.11, we have only notice that
$|(1-\psi ) K(x)| \le C_\psi |x|^{\alpha-Q}$, so that $(1-\psi ) K \in L^{Q/(Q-\alpha),\infty}(\he n)$,
and thereforet i) and ii) hold true.

Suppose now $\alpha=0$.
Notice that $(\psi-1)K \in L^{1,\infty}(\he n)$,
and therefore also $u\to ((\psi-1)K)\ast u$ is $L^p-L^p$ continuous by
Hausdorff-Young Theorem. This proves that iii) holds true.
\end{proof}

\begin{remark}\label{truncation rem}
By Theorem \ref{hls folland}, Lemma \ref{truncation} still holds if we replace $(1-\psi)K$
by $\psi K$.
\end{remark}

 The following (well known) estimate will be useful in the sequel.
 \begin{lemma}\label{pointwise}
 Let $g$ be a a kernel of type $\mu>0$.
 Then, if $f\in \mc D(\he n)$ and $R$ is an homogeneous
 polynomial of degree $\ell\ge 0$ in the horizontal derivatives,
 we have
 $$
R( f\ast g)(p)= O(|p|^{\mu-Q-\ell})\quad\mbox{as }p\to\infty.
 $$
 On the other hand, if  $g$ is
 a smooth function in $\he n\setminus\{0\}$ that
 satisfies the logarithmic estimate
$|g(p)|\le C(1+|\ln|p|| )$ and in addition
its horizontal derivatives are homogeneous of degree $-1$
with respect to group dilations,
 then, if $f\in \mc D(\he n)$ and $R$ is an homogeneous
 polynomial of degree $\ell\ge 0$ in the horizontal derivatives,
 we have
 \begin{eqnarray*}
R( f\ast g)(p)&=&O(|p|^{-\ell})\quad \mbox{as }p\to\infty
\quad\mbox{ if $\ell>0$}; \\
R( f\ast g)(p)&=&O(\ln|p| )\quad \mbox{as }p\to\infty
\quad\mbox{ if $\ell=0$}.
\end{eqnarray*}

 \end{lemma}
 
 Since we have fixed a left-invariant moving frame for $E_0^\bullet$, a $(N_h\times N_k)$-matrix whose entries are kernels of type $\alpha$
 defines in a natural way an operator from $E_0^h$ to $E_0^k$. We still refer to this operator as to an operator
 associated with a (matrix-valued) kernel of type $\alpha$.

 \begin{definition}\label{rumin laplacian} 
In $\he n$, following \cite{rumin_jdg}, we define
the operator $\Delta_{\he{},h}$  on $E_0^h$ by setting
\begin{equation*}
\Delta_{\he{},h}=
\left\{
  \begin{array}{lcl}
     d_c\delta_c+\delta_c d_c\quad &\mbox{if } & h\neq n, n+1;
     \\ (d_c\delta_c)^2 +\delta_cd_c\quad& \mbox{if } & h=n;
     \\d_c\delta_c+(\delta_c d_c)^2 \quad &\mbox{if }  & h=n+1.
  \end{array}
\right.
\end{equation*}

\end{definition}

Notice that $-\Delta_{\he{},0} = \sum_{j=1}^{2n}(W_j^2)$ is the usual sub-Laplacian of
$\he n$.

 For sake of simplicity, since a basis  of $E_0^h$
is fixed, the operator $\Delta_{\he{},h}$ can be identified with a matrix-valued map, still denoted
by $\Delta_{\he{},h}$
\begin{equation}\label{matrix form}
\Delta_{\he{},h} = (\Delta_{\he{},h}^{ij})_{i,j=1,\dots,N_h}: \mc D'(\he{n}, \rn{N_h})\to \mc D'(\he{n}, \rn{N_h}),
\end{equation}
where $\mc D'(\he{n}, \rn{N_h})$ is the space of vector-valued distributions on $\he n$.

This identification makes possible to avoid the notion of currents: we refer to \cite{BFTT} for
a more elegant presentation.

It is proved in \cite{rumin_jdg} that $\Delta_{\he{},h}$ is
hypoelliptic and maximal hypoelliptic in the sense of \cite{HN}. In
general, 
if $\mc L$ is a differential operator  on
$\mc D'(\he{n},\rn {N_h})$, then $\mc L$ is said hypoelliptic if 
for any open set $\mc V\subset \he{n}$ 
where $\mc L\alpha$ is smooth, then $\alpha$ is smooth in $\mc V$.
In addition, if $\mc L$ is
homogeneous of degree $a\in\mathbb N$,
we say that $\mc L$ is maximal hypoelliptic if
 for any  $\delta>0$ there exists $C=C(\delta)>0$ such that for any
homogeneous
polynomial $P$ in $W_1,\dots,W_{2n}$ of degree $a$
we have
$$
\|P\alpha\|_{L^{ 2}(\he n, \rn{N_h})}\le C
\left(
\|\mc L\alpha\|_{L^{ 2}(\he n, \rn{N_h})}+\|\alpha\|_{L^{ 2}(\he n, \rn{N_h})}
\right).
$$
for any $\alpha\in \mc D(B_\rho(0,\delta),\rn {N_h})$.

Combining \cite{rumin_jdg}, Section 3,   and \cite{BFT3}, Theorems 3.1 and 4.1, we obtain the following result.

\begin{theorem}[see \cite{BFP}, Theorem 4.6] \label{global solution}
If $0\le h\le 2n+1$, then the differential operator $\Delta_{\he{},h}$ is
hypoelliptic of order $a$, where $a=2$ if $h\neq n, n+1$ and  $a=4$ 
if $h=n, n+1$ with respect to group dilations. Then
\begin{enumerate}
\item[i)] for $j=1,\dots,N_h$ there exists
\begin{equation}\label{numero}
    K_j =
\big(K_{1j},\dots, K_{N_h j}\big), \quad j=1,\dots N_h
\end{equation}
 with $K_{ij}\in\mc D'(\he{n})\cap \mc
E(\he{n} \setminus\{0\})$,
$i,j =1,\dots,N$;
\item[ii)] if $a<Q$, then the $K_{ij}$'s are
kernels of type $a$
 for
$i,j
=1,\dots, N_h$

 If $a=Q$,
then the $K_{ij}$'s satisfy the logarithmic estimate
$|K_{ij}(p)|\le C(1+|\ln\rho(p)|)$ and hence
belong to $L^1_{\mathrm{loc}}(\he{n})$.
Moreover, their horizontal derivatives  $W_\ell K_{ij}$,
$\ell=1,\dots,2n$, are
kernels of type $Q-1$;
\item[iii)] when $\alpha\in
\mc D(\he{n},\rn {N_h})$,
if we set
\begin{equation}\label{numero2}
    \mc K\alpha:=
\big(    
    \sum_{j}\alpha_j\ast  K_{1j},\dots,
     \sum_{j}\alpha_j\ast  K_{N_hj}\big),
\end{equation}
 then $ \Delta_{\he{},h}\mc K\alpha =  \alpha. $
Moreover, if $a<Q$, also $\mc K\Delta_{\he{},h} \alpha =\alpha$.

\item[iv)] if $a=Q$, then for any $\alpha\in
\mc D(\he{n},\rn {N_h})$ there exists 
$\beta_\alpha:=(\beta_1,\dots,\beta_{N_h})\in \rn{N_h}$,  such that
$$\mc K \Delta_{\he{},h}\alpha - \alpha = \beta_\alpha.$$

%
\end{enumerate}
\end{theorem}

\begin{remark}\label{K}
Coherently with formula \eqref{matrix form}, the operator $\mc K$ can be identified
with an operator (still denoted by $\mc K$) acting on smooth compactly
supported differential forms in $\mc D(\he n, E_0^h)$. Moreover, when the notation
will not be misleading, we shall denote by $\alpha \to \Delta^{-1}_{\mathbb H,h }\alpha$  the convolution with $\mc K$
acting on forms of degree $h$.

\end{remark}

\begin{lemma}\label{comm} 
If $\alpha\in\mc D(\he n, E_0^h)$
\begin{itemize}
\item[i)]$
d_c \Delta^{-1}_{\mathbb H, h}\alpha = \Delta^{-1}_{\mathbb H, h+1} d_c\alpha$, \qquad $h=0,1,\dots, 2n$, 
\qquad $h\neq n-1, n+1$.

\item[ii)] $d_c \Delta^{-1}_{\mathbb H, n-1}\alpha = d_c\delta_c\Delta^{-1}_{\mathbb H, n} d_c\alpha$ \qquad ($h=n-1$).

\item[iii)]$
d_c\delta_c d_c \Delta^{-1}_{\mathbb H, n+1}\alpha = \Delta^{-1}_{\mathbb H, n+2} d_c\alpha$,
 \qquad ($h=n+1$).

\item[iv)]$\delta_c \Delta^{-1}_{\mathbb H, h}\alpha = \Delta^{-1}_{\mathbb H, h-1} \delta_c\alpha$
 \qquad $ h=1,\dots, 2n+1$, \qquad $h\neq n, n+2$.
 
 \item[v)] $\delta_c \Delta^{-1}_{\mathbb H, n+2}\alpha = \delta_c d_c\Delta^{-1}_{\mathbb H, n+1} 
 \delta_c\alpha$ \qquad ($h=n+2$).

\item[vi)]$
\delta_c d_c \delta_c  \Delta^{-1}_{\mathbb H, n}\alpha = \Delta^{-1}_{\mathbb H, n-1} \delta_c \alpha$,
 \qquad ($h=n$).
\end{itemize}
\end{lemma}

\begin{proof} Let us prove i), ii), iii). The remaining assertions will follow by Hodge
duality. Put 
\begin{align*}
\omega_h : &= d_c \Delta^{-1}_{\mathbb H, h}\alpha - \Delta^{-1}_{\mathbb H, h+1} d_c\alpha
\qquad\mbox{if } h\neq n-1, n+1,
\\
\omega_{n-1}: &= d_c \Delta^{-1}_{\mathbb H, n-1}\alpha - d_c\delta_c\Delta^{-1}_{\mathbb H, n} d_c\alpha
\\
\omega_{n+1}: &= d_c\delta_c d_c \Delta^{-1}_{\mathbb H, n+1}\alpha - \Delta^{-1}_{\mathbb H, n+2} d_c\alpha.
\end{align*}
We notice first that, by Theorem \ref{global solution} and Proposition
\ref{kernel}, for all $h=1,\dots,2n$, $\omega_h= M_h\ast \alpha$, where $M_h$ is a kernel
of type 1. 
Thus, by Lemma \ref{pointwise}
\begin{equation}\label{infinity}
\omega_h(x) = O(|x|^{1-Q}) \qquad\mbox{as $ x\to\infty$}.
\end{equation}

We want to show now that
\begin{equation}\label{harmonic}
\Delta^{-1}_{\mathbb H, h+1} \omega_h = 0 \qquad\mbox{for } h=1 ,\dots, 2n.
\end{equation}

Suppose first $h\neq n-1, n, n+1$.  By Theorem \ref{global solution}, we have:
\begin{equation*}\begin{split}
\Delta_{\mathbb H, h+1}  \omega_h & = d_c  \delta_c d_c \Delta^{-1}_{\mathbb H, h}\alpha- d_c\alpha
\\&
= d_c \Delta_{\mathbb H, h} \Delta^{-1}_{\mathbb H, h}\alpha - d_c\alpha = 0.
\end{split}\end{equation*}

If $h=n-1$, then
\begin{equation*}\begin{split}
\Delta_{\mathbb H, n}  \omega_{n-1} & = 
d_c\delta_c d_c \delta_c \Big(
d_c \Delta^{-1}_{\mathbb H, n-1}\alpha - d_c\delta_c\Delta^{-1}_{\mathbb H, n} d_c\alpha\Big)
\\&
=
d_c\delta_c d_c\Delta_{\mathbb H, n-1}  \Delta^{-1}_{\mathbb H, n-1}\alpha
- d_c\delta_c  \Delta^{-1}_{\mathbb H, n} \Delta^{-1}_{\mathbb H, n} d_c\alpha = 0.
\end{split}\end{equation*}

If $h=n$, then (keeping in mind that $d_c \Delta^{-1}_{\mathbb H, n}\alpha$
is a form of degree $n+1$ and  $ \Delta^{-1}_{\mathbb H, n}\alpha$
is a form of degree $n$)
\begin{equation*}\begin{split}
\Delta_{\mathbb H, n+1}  \omega_n & = (   ( \delta_c d_c)^2+          d_c  \delta_c  )d_c \Delta^{-1}_{\mathbb H, n}\alpha- d_c\alpha
\\&
= d_c ( \delta_c  d_c + (d_c\delta_c)^2) \Delta^{-1}_{\mathbb H,n}\alpha - d_c\alpha
\\&
= d_c \Delta_{\mathbb H,n} \Delta^{-1}_{\mathbb H,n}\alpha - d_c\alpha = 0.
\end{split}\end{equation*}

Finally, if $h=n+1$, then
\begin{equation*}\begin{split}
\Delta_{\mathbb H, n+2}  \omega_{n+1} & = 
d_c\delta_c d_c\delta_c d_c \Delta^{-1}_{\mathbb H, n+1}\alpha -  d_c\alpha
\\&
= d_c\Delta_{\mathbb H, n+1}  \Delta^{-1}_{\mathbb H, n+1}\alpha -  d_c\alpha
=0.
\end{split}\end{equation*}

This proves \eqref{harmonic}.

Thus, by  \cite{BFT3}, Proposition 3.2, $\omega $ is a polynomial coefficient form.
Then, by \eqref{infinity} necessarily
 $\omega\equiv 0$. 
 
 This proves i), ii), iii).

%
%
%
%
%
%
\end{proof}

\section{Function spaces}\label{function spaces}


\subsection{Sobolev spaces}

Since here we are dealing only with integer order Folland-Stein function spaces, we
can give this simpler definition (for a general presentation, see e.g. \cite{folland}).

\begin{definition}\label{integer spaces} If $U\subset \he n$ is an open set, $1\le p \le\infty$
and $m\in\mathbb N$, then
the space $W^{m,p}(U)$
is the space of all $u\in L^p(U)$ such that
$$
W^Iu\in L^p(U)\quad\mbox{for all multi-indices $I$ with } d(I)=m,
$$
endowed with the natural norm.

\end{definition}

\begin{theorem} If $U\subset \he n$, $1\le p < \infty$, and $k\in\mathbb N$, then
\begin{itemize}
\item[i)] $ W^{k,p}(U)$ is a Banach space;
\item[ii)] $ W^{k,p}(U)\cap C^\infty (U)$ is dense in $ W^{k,p}(U)$;
\item[iii)] if $U=\he n$, then $\mc D(\he n)$ is dense in $ W^{k,p}(U)$.
\end{itemize}

\end{theorem}




%

\begin{definition} If
$1\le p<\infty$,
we denote  by $\WO{k}{p}{U}$
the completion of $\mc D(U)$ in $W^{k,p}(U)$.
If $U$ is bounded, then
by (iterated) Poincar\'e inequality (see e.g. \cite{jerison}), it follows that the norms
\begin{equation*}
\|u\|_{W^{k,p}(U)} \quad\mbox{and}\quad
\sum_{d(I)=k}\| W^I u\|_{ L^p(U)}
\end{equation*}
are equivalent on $\WO{k}{p}{U}$ when $1\le p<\infty$.
\end{definition}

\medskip

Finally, $W^{k,p}_{\mathrm{Euc}}(U)$
denotes the usual Sobolev space.

%
%
%
%

\subsection{Negative spaces}

\begin{definition}\label{negative spaces}
If $U\subset \he n$ is an open set and $1< p<\infty$,  $W^{-k,p}(U)$ is the dual space of
$\WO{k}{p'}{U}$, where $1/p+1/p'=1$. It is well known that
\begin{equation*}
W^{-k,p}(U)=\{f_0+\sum_{d(I)=k}W^If_I, \;f_0,  f_I\in
 L^p(U) \mbox{ for any $I$ such that } d(I)=k\},
\end{equation*}
and
\begin{equation*}
\|u\|_{W^{-k,p}(U)}\approx
\inf\{ \| f_0 \|_{L^p(U)} + \sum_I \| f_I \|_{L^p(U)}\,;\,
d(I)=k,  f_0+\sum_{d(I)=k}W^If_I =u \}.
\end{equation*}

If $U$ is bounded, then we can take $f_0=0$.

Finally, we stress  that
\begin{equation*}
\{f_0+\sum_{d(I)=k}W^If_I, \;f_0,  f_I\in
\mc D(U) \mbox{ for any $I$ such that } d(I)=k\}
\end{equation*}
is dense in $W^{-k,p}(U)$.
\end{definition}

\begin{definition} \label{dual spaces forms} If $U\subset \he n$ is an open set, $0\le h\le 2n+1$, $1\le p\le \infty$ and $m\ge 0$,
we denote by $W^{m,p}(U,\cov{h})$ (by $\WO{m}{p}{U,\cov{h}}$)
the space of all sections of $\cov{h}$ such that their
components with respect to a given left-invariant frame  belong to
$W^{m,p}(U)$ (to $\WO{m}{p}{U}$, respectively), endowed with its natural norm. Clearly, this definition
is independent of  the choice of the frame itself.

The spaces $W^{m,p}(U,E_0^{h})$ and $\WO{m}{p}{U,E_0^h}$ are defined in the same way.

On the other hand, the spaces 
$$
W^{-m,p}(U,E_0^{h}):= \Big(\WO{m}{p'}{U,E_0^h}\Big)^*
$$ can be viewed as spaces of currents
on $(E_0^\bullet,d_c)$ as in \cite{BFTT},  Proposition 3.14.  Again as in \cite{BFTT},  Proposition 3.14,
an element of $W^{-m,p}(U,E_0^{h})$ can be identified (with respect to our basis) with
a $N_h$-ple 
$$
(T_1,\dots,T_{N_h}) \in \Big( W^{-m,p}(U,E_0^{h})\Big)^{N_h}
$$
(this is nothing but the intuitive notion of ``currents as differential form with distributional coefficients''). The action 
of $u\in W^{-m,p}(U,E_0^{h})$ associated with $(T_1,\dots,T_{N_h})$ on the form $\sum_j \alpha_j\xi_j^h\in \WO{m}{p'}{U,E_0^h}$
is given by
$$
\Scal{u}{\alpha}:= \sum_j \Scal{T_j}{\alpha_j}.
$$
On the other hand, suppose for sake of simplicity that $U$ is bounded, then 
by  Definition \ref{negative spaces} there exist $f_I^j \in L^p(U)$,
$j=1,\dots, N_h$, $i=1,\dots, 2n+1$ such that
\begin{equation}\label{dual spaces forms eq:1}
 \Scal{u}{\alpha} = \sum_j \sum_{d(I)=m} \int_U f_I^j(x) W^I\alpha_j(x)\, dx.
\end{equation}

\end{definition}

Alternatively, one can express duality in spaces of differential forms using the pairing between $h$-forms and $2n+1-h$-forms defined by
\begin{eqnarray*}
\alpha,\beta\mapsto\int_{U}\alpha\wedge\beta.
\end{eqnarray*}
Note that this makes sense for Rumin forms and is a nondegenerate pairing. In this manner, the dual of $L^p(U,E_0^h)$ is $L^{p'}(U,E_0^{2n+1-h})$. Hence $W^{-m,p}(U,E_0^h)$ consists of differential forms of degree $2n+1-h$ whose coefficients are distributions belonging to $W^{-m,p}(U)$.

\subsection{Contact invariance}

\begin{lem}\label{pullback}
Let $U$, $V$ be open subsets of $\he n$. Let $\phi:U\to V$ be a $C^k$-bounded contact diffeomorphism. Let $\ell=-k+1,\ldots,k-1$. Then the pull-back operator $\phi^\sharp$ from $W^{\ell,p}$ forms on $V$ to $W^{\ell,p}$ forms on $U$ is bounded, and its norm depends only on the $C^k$ norms of $\phi$ and $\phi^{-1}$.
\end{lem}
When $\ell\geq 0$, this follows from the chain rule and the change of variables formula. According to the change of variables formula
\begin{eqnarray*}
\int_{U}\phi^{\sharp}\alpha\wedge \phi^{\sharp}\beta=\int_{V}\alpha\wedge\beta,
\end{eqnarray*}
the adjoint of $\phi^{\sharp}$ with respect to the above pairing is $(\phi^{-1})^{\sharp}$. Hence $\phi^{\sharp}$ is bounded on negative Sobolev spaces of differential forms as well.

\subsection{Sobolev spaces on contact sub-Riemannian manifolds}

We define Sobolev spaces (involving  a positive or negative number of derivatives) on bounded geometry contact sub-Riemannian manifolds.

Let $(M,H,g)$ be a bounded $C^k$-geometry sub-Riemannian contact manifold. Pick a uniform covering $\mathcal{U}$ by equal radius balls (uniform means that distances between centers are bounded below). Let $\phi_j:B\to U_j$ be $C^k$-bounded contact charts from the unit Heisenberg ball. Given a differential form $\omega$ on $M$, let $\omega_j=\phi_j^{\sharp}\omega$. Let $-k+1\leq \ell\leq k-1$ be an integer. Define
\begin{eqnarray*}
\|\omega\|_{\mathcal{U},\ell,p}=\left(\sum_{j}\|\omega_j\|_{W^{\ell,p}(B)}^p\right)^{1/p}.
\end{eqnarray*}

Let us show that an other uniform covering $\mathcal{U}'$ and other choices of controlled charts lead to an equivalent norm. Every piece $U$ of $\mathcal{U}$ is covered with boundedly many pieces $U'_i$ of $\mathcal{U}'$. Thus
\begin{eqnarray*}
\|\omega_j\|_{W^{\ell,p}}^p\leq \sum\|{\omega_j}_{|{\phi_j}^{-1}(U'_i)}\|_{W^{\ell,p}({\phi_j}^{-1}(U'_i))}^p.
\end{eqnarray*}
Since ${\omega_j}_{|{\phi_j}^{-1}(U'_i)}$ is the pull-back by the contactomorphism $\phi={\phi_j}\circ{\phi'_i}^{-1}$ of ${\omega_i}_{|{\phi'_i}^{-1}(U_j)}$, Lemma \ref{pullback} implies that
\begin{eqnarray*}
\|{\omega_j}_{|{\phi_j}^{-1}(U'_i)}\|_{W^{\ell,p}({\phi_j}^{-1}(U'_i))}\leq C\,\|{\omega_i}_{|{\phi'_i}^{-1}(U_j)}\|_{W^{\ell,p}({\phi'_i}^{-1}(U_j))},
\end{eqnarray*}
where the constant only depends on the uniform bound on horizontal derivatives of order $\leq k$ of $\phi$. Thus
\begin{eqnarray*}
\|\omega_j\|_{W^{\ell,p}}^p\leq \sum\|{\omega'_i}\|_{W^{\ell,p}(B)}^p.
\end{eqnarray*}
When summing over $j$, each term $\|{\omega'_i}\|$ on the right hand side occurs only a bounded number $N$ of times. This yields
\begin{eqnarray*}
\|\omega\|_{\mathcal{U},\ell,p}\leq CN^{1/p}\|\omega\|_{\mathcal{U'},\ell,p}.
\end{eqnarray*}

\section{Homotopy formulae and Poincar\'e and Sobolev inequalities}
\label{poincare}

In this paper we are mainly interested to obtain functional inequalities for differential forms that are the counterparts
of the classical $(p,q)$-Sobolev and Poincar\'e inequalities on a ball $B\subset \mathbb R^n$ with sharp exponents of the form
$$
\| u - u_{B}\|_{L^q}(B) \le C(r) \|\nabla u\|_{L^p}(B) 
$$
(as well as of its counterpart for compactly supported functions). In this case, we can choose $q= pn/(n-p)$, provided 
$p<n$. 

\begin{definition}\label{poincare def}
Take $\lambda >1$ and set  $B=B(e,1)$ and $B'=B(e,\lambda)$,
where the $B(x,r)$'s are the Kor\'anyi balls  in $\mathbb H^n$ (in particular
the balls centered at $x=e$, and then all balls, are convex).
 If $1\le k\le 2n+1$  and $q\ge p\ge 1$, we say that the interior $(p,q)$-Poincar\'e inequality holds in $E_0^k$ if 
there exists a constant $C$ such that,
 for every $d_c$-closed differential $k$-form $\omega$ in $L^p(B';E_0^k)$ there exists a differential $k-1$-form $\phi$ in $L^q(B,E_0^{k-1})$ such that $d_c\phi=\omega$ and
\begin{eqnarray*}
\|\phi\|_{L^{q}(B,E_0^{k-1})}\leq C\,\|\omega\|_{L^{p}(B',E_0^k)} \quad 
\mbox{ 
$\mathrm{interior}\, \he{}$-$\mathrm{Poincar\acute{e}}_{p,q}(k))$.
}
\end{eqnarray*}

\end{definition}

\begin{remark}\label{poincare k=1} If $k=1$ and $Q>p\ge 1$, then $(\, \he{}$-$\mathrm{Poincar\acute{e}}_{p,q}(1))$
is nothing but the usual Poincar\'e inequality with  $\displaystyle \frac{1}{p}-\frac{1}{q}= \frac{1}{Q}$
(see e.g. \cite{FLW_grenoble}, \cite{capdangar}, \cite{MSC}).
\end{remark}

\begin{remark} If we replace Rumin's complex $(E_0^\bullet, d_c)$ by the usual de Rham's complex
$(\Omega^\bullet,d)$ in $\rn{2n+1}$, then the $(p,q)$-Poincar\'e inequality holds on Euclidean balls for $k=1$
and $n>p\ge1$. If $k>1$, then the $(p,q)$-Poincar\'e inequality for $2n+1>p>1$ and $\displaystyle 
\frac{1}{p}-\frac{1}{q}= \frac{1}{2n+1}$ is proved by Iwaniec \& Lutoborski (see \cite{IL}, Corollary 4.2).
\end{remark}

The $\he{}$-$\mathrm{Poincar\acute{e}}_{p,q}(k)$ inequality (as well as its Euclidean counterpart)
can be formulated by duality as follows.

\begin{definition} \label{equiv Sobolev}
Take $\lambda >1$ and set  $B=B(e,1)$ and $B'=B(e,\lambda)$. If $1\le k\le 2n$, $1\le p\le q < \infty $ and $q\ge p$,
we say  that the (local) $\he{}$-$\mathrm{Sobolev}_{p,q}(k)$ inequality holds if  there exists a constant $C$ such that 
 for every compactly supported smooth $d_c$-closed differential $k$-form $\omega$ in $L^p(B;E_0^k)$ there exists a 
smooth compactly supported differential $(k-1)$-form $\phi$ in $L^q(B',E_0^{k-1})$ such that $d_c\phi=\omega$ in $B'$
and
\begin{eqnarray}\label{H Sobolev}
\|\phi\|_{L^{q}(B',E_0^{k-1})}\leq C\,\|\omega\|_{L^{p}(B,E_0^k)}. \quad 
\end{eqnarray}

\end{definition}

Notice that, in this case, we do not distinguish interior inequalities (in other words, we can always assume $B=B'$), basically since, when dealing with compactly
supported forms, the structure of the boundary does not affect the estimates.

\begin{remark}\label{sobolev k=1} If $k=1$ and $Q>p\ge 1$, then $(\, \he{}$-$\mathrm{Sobolev}_{p,q}(1))$
is nothing but the usual Sobolev  inequality with  $\displaystyle \frac{1}{p}-\frac{1}{q}= \frac{1}{Q}$.

\end{remark}

In \cite{IL}, starting from Cartan's homotopy formula, the authors proved  that,
if $D \subset \rn {N}$ is a convex set,
$1<p<\infty$, $1<k<N$, then  there exists a linear bounded map: 
\begin{equation}\label{Keuc}
K_{\mathrm{Euc},k}  : L^p(D, {\bigwedge}\vphantom{!}^k)\to W^{1,p}(D, {\bigwedge}\vphantom{!}^{k-1})
\end{equation}
 that
is a homotopy operator, i.e.
\begin{equation}\label{may 4 eq:1}
	\omega =  dK_{\mathrm{Euc},k} \omega + K_{\mathrm{Euc},k+1}d\omega \qquad \mbox{for all 
	$\omega\in C^\infty (D, {\bigwedge}\vphantom{!}^k)$}
\end{equation}
(see Proposition 4.1 and Lemma 4.2 in \cite{IL}). More precisely, $K_{\mathrm{Euc}}$ has the form
\begin{equation}\label{10 maggio eq:1}
K_{\mathrm{Euc},k} \omega (x)= \int_D \psi(y)K_y\omega(x) \, dy,
\end{equation}
where $\psi \in \mc D(D)$, $\int_D\psi(y)\, dy=1$, and 
\begin{equation}\label{10 maggio eq:2}\begin{split}
&\Scal{K_y\omega(x)}{\xi_1\wedge\cdots
\wedge \xi_{k-1})}:=\int_0^1t^{k-1}\Scal{\omega(tx+(1-t)y)}{(x-y)\wedge\xi_1\wedge\cdots
\wedge \xi_{k-1})}.
\end{split}\end{equation}

Starting from \cite{IL}, in \cite{mitrea_mitrea_monniaux}, Section 4,  the authors define a compact homotopy operator $J_{\mathrm{Euc},k}$ in Lipschitz star-shaped  domains in the Euclidean space $\rn {N}$, providing an explicit representation formulas 
for $J_{\mathrm{Euc},k}$, together with continuity properties among Sobolev spaces. More precisely, if $D\subset \rn {N}$ is a star-shaped Lipschitz domain and $1<k<N$, then  there exists
$$
J_{\mathrm{Euc},k} : L^{p}(D, {\bigwedge}\vphantom{!}^k) \to W^{1,p}_{0}(D, {\bigwedge}\vphantom{!}^{k-1})
$$
such that
$$
\omega = dJ_{\mathrm{Euc},k}\omega + J_{\mathrm{Euc},k+1}d\omega \qquad \mbox{for all $\omega\in 
\mc D(D, {\bigwedge}\vphantom{!}^k)$.}
$$


Take now $D=B(e,1)=:B$ and $N=2n+1$. If $\omega\in C^\infty(B,E_0^k)$, then we set
\begin{eqnarray}\label{may 4 eq:2}
K=\Pi_{E_0}\circ \Pi_E \circ K_{\mathrm{Euc}}   \circ \Pi_E
\end{eqnarray}
(for sake of simplicity, from now on we drop the index $k$ - the degree of the form -
writing, e.g., $K_{\mathrm{Euc}}$ instead of $K_{\mathrm{Euc},k}$.

Analogously, we can define
\begin{eqnarray}\label{may 31 eq:2}
J=\Pi_{E_0}\circ \Pi_E \circ J_{\mathrm{Euc}}   \circ \Pi_E.
\end{eqnarray}

Then $K$ and $J$ invert Rumin's differential $d_c$ on closed forms of the same degree. More
precisely, we have:

\begin{lemma}\label{homotopy 1} If $\omega$ is $d_c$-closed, then
\begin{equation}\label{homotopy closed}
\omega = d_cK\omega \quad\mbox{if $1\le k\le 2n+1$}\qquad\mbox{and}\qquad   
\omega = d_cJ\omega \quad\mbox{if $1\le k\le 2n$.}
\end{equation}
In addition, if $\omega$ is compactly supported in $B$, then $J\omega$ is still compactly supported in $B$.
\end{lemma} 

\begin{proof} Consider for instance $d_c K\omega$. 
If $d_c\omega =0$,
then $d(\Pi_E\omega)=0$, and hence
$$
\Pi_E\omega = dK_{\mathrm{Euc}}  (\Pi_E\omega),
$$
by \eqref{may 4 eq:1}.
By \eqref{may 4 eq:2} (and recalling that $d \Pi_E= \Pi_Ed$ and $\Pi_E\Pi_{E_0} \Pi_E=\Pi_E$), 
\begin{eqnarray*}
d_cK\omega&=&\Pi_{E_0}d \Pi_E\Pi_{E_0} \Pi_E  K_{\mathrm{Euc}}    \Pi_E\omega=\Pi_{E_0}d \Pi_E K_{\mathrm{Euc}}    \Pi_E\omega
\\
&=&\Pi_{E_0} \Pi_E d K_{\mathrm{Euc}}    \Pi_E\omega=\Pi_{E_0} \Pi_E \Pi_E\omega
=\Pi_{E_0} \Pi_E \Pi_{E_0}\omega=\omega\,.
\end{eqnarray*}
Finally, if $\mathrm{supp}\,  \omega
\subset B$, then $\mathrm{supp}\, J \omega
\subset B$ since both $\Pi_E$ and $\Pi_{E_0}$ preserve the support.
\end{proof}

\begin{lemma}\label{senza nome} Put $B=B(e,1)$. Then:
\begin{itemize}
\item[i)] if $1< p < \infty$ and $k=1,\dots, 2n+1$, then $K :W^{1,p}(B, E_0^k) \to
L^p(B, E_0^{k-1})$ is bounded;
\item[ii)] if $1\le p\le \infty$ and $n+1<k \le 2n+1$, then $K :L^{p}(B, E_0^k) \to
L^p(B, E_0^{k-1})$ is compact;
\item[iii)] if $1< p< \infty$ and $k=n+1$, then $K:L^{p}(B, E_0^{n+1}) \to
L^p(B, E_0^{n})$ is bounded.
\end{itemize}
Analogous assertions hold for $1\le k\le 2n$ when we replace $K$ by $J$. In addition, $\mathrm{supp}\, J \omega
\subset B$.
\end{lemma}

\begin{proof} 
By its very definition, $\Pi_E: W^{1,p}(B, E_0^k) \to
L^p(B, E_0^{k})$ is bounded. By \eqref{Keuc}, $K_{\mathrm{Euc}}$ is continuous
from  $L^p(B, E_0^{k})$ to $W^{1,p}(B, E_0^{k-1})$ and hence, in particular,
from  $L^p(B, E_0^{k})$ to $W^{1,p}(B, E_0^{k-1})$. Then we can conclude
 the proof of i),
 keeping again into account that $\Pi_E$ is a differential operator of order $\le 1$
 in the horizontal derivatives.

To prove ii) it is enough to remind that $K = \Pi_{E_0}K_{\mathrm{Euc}}$ of forms of degree $h>n$, together with 
Remark 4.1 in \cite{IL}.

As for iii), the statement can be proved similarly to i), noticing that $K = \Pi_{E_0}\Pi_E K_{\mathrm{Euc}}$ on forms of degree $n+1$.

Finally, $\mathrm{supp}\, J \omega
\subset B$ since both $\Pi_E$ and $\Pi_{E_0}$ preserve the support.

\end{proof}

The operators $K$ and $J$ provide a local homotopy in Rumin's complex, but fail to yield the
Sobolev and Poincar\'e inequalities we are looking for, since, because of the presence
of the projection operator $\Pi_E$ (that on forms of low degree is a first order differential
operator) they loose regularity as is stated in Lemma \ref{senza nome}, ii) above. In order
to build ``good'' local homotopy operators with the desired gain of regularity, we have
to combine them with  homotopy operators which, though not local, in fact provide the ``good''
gain of regularity.

\begin{proposition}\label{homotopy formulas}  If $\alpha\in \mc D(\he n, E_0^h)$
 for $p>1$ and $h= 1,\dots,2n$, then the following homotopy formulas hold:
\begin{itemize}
\item if $h\neq n, n+1$, then $\alpha = d_c K_1 \alpha + \tilde K_1d_c \alpha $, where $K_1$ and $\tilde K_1$ are  associated with kernels $k_1, \tilde k_1$ of type 1;
\item  if $h = n$, then $\alpha = d_c K_1 \alpha + \tilde K_2 d_c\alpha$, where $K_1$ and $\tilde K_2$ are  associated with kernels $k_1, \tilde k_2$ of type 1 and 2,
respectively;
\item  if $h = n+1 $, then $\alpha = d_c K_2 \alpha + \tilde K_1 d_c\alpha$, where $K_2$ and $\tilde K_1$ are  associated with kernels $k_2, \tilde k_1$ of type 2
and $1$, respectively.
\end{itemize}

\end{proposition}

 \begin{proof} 
Suppose $h\neq n-1,n,n+1$. By Lemma \ref{comm}, we have:
\begin{equation*}\begin{split}
\alpha &=  
 \Delta_{\mathbb H,h}  \Delta^{-1}_{\mathbb H,h}\alpha = d_c (\delta_c  \Delta^{-1}_{\mathbb H,h})\alpha
+ \delta_c (d_c  \Delta^{-1}_{\mathbb H,h})\alpha
\\&
=d_c (\delta_c  \Delta^{-1}_{\mathbb H,h})\alpha + 
(\delta_c   \Delta^{-1}_{\mathbb H,h+1})d_c \alpha.
\end{split}\end{equation*}
where $\delta_c  \Delta^{-1}_{\mathbb H,h}$ and $\delta_c   \Delta^{-1}_{\mathbb H,h+1}$ are
 associated with a kernel of type $1$ (by Proposition \ref{kernel}
and Theorem \ref{global solution}).

Analogously, if $h=n-1$
\begin{equation*}\begin{split}
\alpha &=  
 \Delta_{\mathbb H,n-1}  \Delta^{-1}_{\mathbb H,n-1}\alpha = d_c (\delta_c  \Delta^{-1}_{\mathbb H,n-1})\alpha
+ \delta_c (d_c  \Delta^{-1}_{\mathbb H,n-1})\alpha
\\&
=d_c (\delta_c  \Delta^{-1}_{\mathbb H,n-1})\alpha + 
(\delta_c   d_c\delta_c\Delta^{-1}_{\mathbb H, n} )d_c \alpha.
\end{split}\end{equation*}
Again $\delta_c  \Delta^{-1}_{\mathbb H,n-1}$ and $\delta_c   d_c\delta_c\Delta^{-1}_{\mathbb H, n}$
are
 associated with  kernels of type $1$.

Take now $h=n$. Then
\begin{equation*}\begin{split}
\alpha &=   \Delta_{\mathbb H,n}  \Delta^{-1}_{\mathbb H,n}\alpha = (d_c \delta_c )^2 \Delta^{-1}_{\mathbb H,n} \alpha
+ \delta_c (d_c  \Delta^{-1}_{\mathbb H,n})\alpha
\\&
= d_c (\delta_c d_c \delta_c  \Delta^{-1}_{\mathbb H,n})\alpha + 
\delta_c   \Delta^{-1}_{\mathbb H,n+1}d_c \alpha
\end{split}\end{equation*}
where $\delta_c d_c \delta_c \Delta^{-1}_{\mathbb H,n}$ and $\delta_c   \Delta^{-1}_{\mathbb H,n+1}$ are associated with a kernel of type $1$ and $2$,
respectively).

Finally, take  $h=n+1$. Then
\begin{equation*}\begin{split}
\alpha &=   \Delta_{\mathbb H,n+1}  \Delta^{-1}_{\mathbb H,n+1}\alpha = d_c \delta_c \Delta^{-1}_{\mathbb H,n+1} \alpha
+ (\delta_c d_c )^2 \Delta^{-1}_{\mathbb H,n+1} \alpha
\\&
= d_c \delta_c \Delta^{-1}_{\mathbb H,n+1} \alpha + 
\delta_c \Delta^{-1}_{\mathbb H,n+2} d_c\alpha
\end{split}\end{equation*}
where $\delta_c \Delta^{-1}_{\mathbb H,n+1} $ and  $\delta_c \Delta^{-1}_{\mathbb H,n+2}$  associated with  kernels of type $2$ and $1$, respectively.

\end{proof}

The $L^p-L^q$ continuity properties of convolution operators associated with Folland's kernels yields the following
strong $\he{}$-$\mathrm{Poincar\acute{e}}_{p,q}(h)$ inequality in $\he n$ (the strong $\he{}$-$\mathrm{Sobolev}_{p,q}(h)$
is obtained in Corollary \ref{strong sobolev}).

\begin{corollary}\label{strong poincare}
Take $1\le h\le 2n+1$. Suppose  $1<p<Q$ if  $h\not=n+1$ and $1<p<Q/2$ if $h=n+1$. Let $q\ge p$ defined by
\begin{eqnarray}\label{kpq 3}
 \frac{1}{p}-\frac{1}{q} := \begin{cases}
\frac{1}{Q}      & \text{ if }h\not=n+1, \\
\frac{2}{Q}      & \text{ if }h=n+1.
\end{cases}
\end{eqnarray}
Then for any $d_c$-closed form $\alpha\in \mc D(\he n, E_0^h)$ there exists $\phi\in L^q(\he, E_0^{h-1})$
such that $d_c\phi=\alpha$ and 
$$
\|\phi\|_{L^q(\he n, E_0^{h-1})} \le C \|\alpha \|_{L^p(\he n, E_0^{h-1})}
$$
(i.e., the strong $\he{}$-$\mathrm{Poincar\acute{e}}_{p,q}(h)$ 
inequality holds for $1\le h\le 2n+1$).

\end{corollary}

\begin{theorem}\label{smoothing}
Let $B=B(e,1)$ and $B'=B(e,\lambda)$, $\lambda>1$, be concentric balls of $\he{n}$. If $1\le h\le 2n+1$, there exist operators $T$ and $\tilde T$  from 
$C^\infty(B', E_0^\bullet)$  to $C^\infty(B, E_0^{\bullet-1})$ and $S$ from 
$C^\infty(B', E_0^\bullet)$  to $C^\infty(B, E_0^\bullet)$ satisfying
\begin{equation}\label{approx homotopy tilde}
d_c T+ \tilde Td_c + S=I\qquad\mbox{on $B$.} 
\end{equation}

In addition
 \begin{itemize}
 \item[i)] $\tilde T: W^{-1,p}(B',E_0^{h+1}) \to L^p(B,E_0^{h})$ if $h\neq n$, and $\tilde T: W^{-2,p}(B,E_0^{n+1}) \to L^p(B,E_0^{n})$;
 \item[ii)] $T: L^p (B',E_0^{h}) \to W^{1,p}(B,E_0^{h-1})$,   $h\neq n+1$, $ T: T: L^p(B',E_0^{n+1}) \to W^{2,p}(B,E_0^{n})$ if $h=n+1$,
 \item[iii)] $S:L^p(B',E_0^h)\to W^{s,p}(B,E_0^{h})$,
\end{itemize}
so that  \eqref{approx homotopy tilde} still holds in $L^p(B, E_0^\bullet)$.
In addition,
for every $(h,p,q)$ satisfying inequalities 
\begin{eqnarray}\label{kpq}
1 <p\leq q <\infty,\quad \frac{1}{p}-\frac{1}{q}\leq  \begin{cases}
\frac{1}{Q}      & \text{ if }h\not=n+1, \\
\frac{2}{Q}      & \text{ if }h=n+1,
\end{cases}
\end{eqnarray}
we have:
\begin{itemize}
 \item[iv)] $T: L^p (B',E_0^{h}) \to L^{q}(B,E_0^{h-1})$;
 \item[v)] $S:L^p(B',E_0^h)\to W^{s,q}(B,E_0^{h})$;
 \item[vi)] $W^{1,p}(B',E_0^h)\to W^{s,q}(B,E_0^{h- 1})$
for any $s>0$.

\end{itemize}

\end{theorem}

\begin{proof} Suppose first $h\neq n,n+1$. We consider a cut-off function $\psi_R$ supported in a $R$-neighborhood
of the origin, such that $\psi_R\equiv 1$ near the origin. With the notations of Proposition \ref{homotopy formulas}, we can write  
$k_1=k_1\psi_R + (1-\psi_R)k_1$ and $\tilde k_1=\tilde k_1\psi_R + (1-\psi_R)\tilde k_1$. 
Let us denote by $K_{1,R}$, $\tilde K_{1,R}$ the convolution operators associated with
$\psi_R k_1$, $\psi_R \tilde k_1$, respectively.
Le us fix two balls $B_0$, $B_1$ with 
\begin{equation}\label{varie B}
B\Subset B_0 \Subset B_1\Subset B',
\end{equation}
and a cut-off function $\chi \in \mc D(B_1)$,
 $\chi\equiv 1$ on $B_0$. If $\alpha \in C^\infty(B', E_0^\bullet)$, we set $\alpha_0=  \chi\alpha$, continued by zero outside $B_1$.

Keeping in mind \eqref{convolution by parts} and Proposition \ref{kernel}, we have
\begin{equation}\label{sept 9 eq:1}
\alpha_0 = d_c K_{1,R} \alpha_0 + \tilde K_{1,R}d_c \alpha_0 + S_0\alpha_0,
\end{equation}
where $S_0$ is 
$$
S_0\alpha_0 :=  d_c( (1-\psi_R)k_1\ast \alpha_0) + (1-\psi_R)\tilde k_1 \ast d_c\alpha_0.
$$ 
 We set
$$
T\alpha := K_{1,R} \alpha_0, \qquad \tilde T\alpha:=  \tilde K_{1,R}d_c \alpha_0, \qquad S\alpha:=  S_0\alpha_0.
$$
We notice that, provided $R>0$ is small enough, the definition of $T$ and $\tilde T$
 does not depend on the continuation of $\alpha$  outside $B_0$.
By \eqref{sept 9 eq:1} we have 
 $$
\alpha = d_c T\alpha + \tilde Td_c \alpha + S\alpha \qquad\mbox{in $B$}.
$$
If $h=n$ we can carry out the same construction, replacing $\tilde k_1$ by $\tilde k_2$ (keep in mind that $\tilde k_2$
is a kernel of type 2). Analogously, if $h=n+1$ we can carry out the same construction, replacing $k_1$ by $ k_2$ (again
a kernel of type 2). 

Let us prove i). Suppose $h\neq n$, and take $\beta\in W^{-1,p}(B', E_0^{h})$. The operator
 $\tilde K_{1,R}$ is associated with a matrix-valued kernel $\psi_R ( \tilde{k}_1)_{\ell,\lambda}$ and
 $\beta$ is identified with a vector-valued distribution $(\beta_1, \dots, \beta_{N_h})$,
 with $\beta_j = \sum_iW_if^j_i$ as in Definition \ref{dual spaces forms} with
 $$
 \sum_j \sum_i \| f_i^j\|_{L^p(B')} \le C \| \beta \|_{W^{-1,p}(B', E_0^h)}.
 $$
 Thus $(\beta_0)_j$, the $j$-th component of $\beta_0 = \chi\beta$ has the form
 $$
 (\beta_0)_j = \sum_iW_i(\chi f_i^j) - \sum_i (W_i\chi) f^j_i=: \sum_iW_i(f^j_i)_0 - \sum_i (W_i\chi) f^j_i.
 $$

 In order to estimate the  norm of $\tilde T \beta$ in $L^p(B, E_0^h)$, we take
 $$
 \phi = \sum_j \phi_j  \xi_j^h \in \mc D(B, E_0^h),\qquad\mbox{with\qquad  $\sum_j\| \phi_j \|_{L^{p'}(B')} \le 1$,}
 $$
 and we estimate $\Scal{T\beta}{\phi}$, that, 
 by \eqref{dual spaces forms eq:1}, is a sum of  terms of the form
\begin{equation}\label{caso 1}
 \int_{B} ( \psi_R \kappa\ast f_0) (x) W_i\phi(x)\, dx
 = \Scal{\psi_R \kappa\ast W_i f_0}{\phi}
\end{equation}
 or of the form 
\begin{equation}\label{caso 2}
 \int_{B} ( \psi_R \kappa\ast (W_i\chi) f) (x) \phi(x)\, dx,
\end{equation}
where $\kappa$ denotes one of the kernels $( \tilde{k}_1)_{\ell,\lambda}$ of type 1 associated with $\tilde{k}_1$, $f$ is one of the $f_i^j$'s and $\phi$ one of the $\phi_j$'s,

As for \eqref{caso 1}, by \eqref{convolution by parts}, 
\begin{equation*}\begin{split}
&\Scal{\psi_R \kappa\ast  W_i f_0}{\phi} = \Scal{ \ccheck W^I\,\ccheck [\psi_R \kappa]\ast  f_0}{\phi}
\\& \hphantom{xxxxx} =
\Scal{\psi_R \ccheck W^I\,\ccheck \kappa\ast  f_0}{\phi} - \Scal{( \ccheck W^I\,\ccheck\psi_R) \kappa\ast  f_0}{\phi}
\end{split}\end{equation*}
We notice now that $ \ccheck W^I\,\ccheck \kappa$ is a kernel of type 0.  Therefore, by Lemma \ref{truncation}
\begin{equation*}\begin{split}
& \Scal{\psi_R \ccheck W^I\,\ccheck \kappa\ast  f_0}{\phi} 
\le \| \psi_R \ccheck W^I\,\ccheck \kappa\ast  f_0\|_{L^p(B)} \|\phi\|_{L^{p'}(B)}
\\& \hphantom{xxxxx} \le 
\| \psi_R \ccheck W^I\,\ccheck \kappa\ast  f_0\|_{L^p(B)}
\le
C \|  f_0\|_{L^p(B')}
\\&\hphantom{xxxxx}
 \le C \| \beta \|_{W^{-1,p}(B', E_0^h)}.
\end{split}\end{equation*}
 The term in \eqref{caso 2} can be handled in the same way, keeping into account Remark \eqref{truncation rem}.
 Eventually, combining \eqref{caso 1} and \eqref{caso 2} we obtain that
 $$
 \|\tilde T\beta\|_{L^p(B)} \le C  \| \beta \|_{W^{-1,p}(B', E_0^h)}.
 $$
 
 The assertion for $h=n$ can be proved in the same way, taking into account that $\tilde T$ is built from a kernel of
 type 2, and that the space $W^{-2,p}(B,E_0^{n+1})$ is characterized by ``second order divergences''.
 
 Let us prove now ii). Suppose $h\neq n+1$ and take $\alpha = \sum_j \alpha_j  \xi_j^h \in \mc D(B', E_0^h)$. Arguing as above, in
 order to estimate $\| T\alpha\|_{W^{1,p}(B,E_0^{h-1})}$ we have to consider terms of the form
 \begin{equation}\label{caso 3}
W_\ell(\psi_R \kappa \ast (\chi\alpha_j) ) = \psi_R \kappa \ast (W_\ell(\chi\alpha_j) )
\end{equation}
(when we want to estimate the the $L^p$-norm of the horizontal derivatives of $T\alpha$),
or of the form
 \begin{equation}\label{caso 4}
\psi_R \kappa \ast (\chi\alpha_j) 
\end{equation}
(when we want to estimate the $L^p$-norm of $T\alpha$).
Both \eqref{caso 3} and \eqref{caso 4} can be handled as in the case i) (no need here of the
duality argument).

We point out that \eqref{caso 4} yields a $L^p-L^q$ estimates (since, unlike \eqref{caso 3}, involves only
kernels of type 1) and then assertion iv) follows.

Let us prove v). Then also iii) will follow straightforwardly.

 It is easy to check that $S_0$ can be written as a convolution operator with matrix-valued kernel $s_0$.
In turn, each entry of $s_0$ (that we still denote by $s_0$) is a sum of terms of the form
$$
(1-\psi_R) W_\ell \kappa - (W_\ell \psi_R)\kappa.
$$
Thus, the kernels are smooth and then regularizing from $\mc E'(B')$ to $C^\infty$ of a neighborhood of $B$.
Thus
$$
\| W^I s_0\ast \alpha_j\|_{L^{q}(B)}
\le C\|\alpha_j\|_{L^{p}(B)},
$$
for all $p,q$.

\end{proof}

\begin{remark}\label{tenda} Apparently, in previous theorem, two different homotopy operators $T$ and $\tilde T$
appear. In fact, they coincide when acting on form of the same degree.

More precisely, in Proposition \ref{homotopy formulas} the homotopy formulas involve four operators
$K_1, \tilde K_1, K_2, \tilde K_2$, where the notation is meant to distinguish operators acting on $d_c\alpha$ (the operators
with tilde) from those on which the differential acts (the operators
without tilde), whereas the lower index 1 or 2 denotes the type of the associated kernels. Alternatively, a different notation could be used:
if $\alpha\in \mc D(\he n, E_0^h)$ we can write 
$$
\alpha = d_c K_h + \tilde K_{h+1}d_c\alpha,
$$
where the tilde has the same previous meaning, whereas the lower index refers now to the degree of the forms on which
the operator acts.

It is important to notice that
$$
K_{h+1} = \tilde K_{h+1}, \qquad h=1,\dots, 2n.
$$
Indeed, take $h<n-1$. Then $\tilde K_{h+1} = \delta_c \Delta^{-1}_{\mathbb H,h+1}$ (as it appears
in the homotopy formula at the degree $h$), that equals $K_{h+1}$ (as it appears
in the homotopy formula at the degree $h+1\le n-1$). Take now $h=n-1$. Then
$\tilde K_{n} = \delta_c d_c\delta_c\Delta^{-1}_{\mathbb H,n}$ (as it appears
in the homotopy formula at the degree $n$), that equals $K_{n}$ (as it appears
in the homotopy formula at the degree $n$). If $h=n$, then 
$\tilde K_{n+1} = \delta_c \Delta^{-1}_{\mathbb H,n+1}$ (as it appears
in the homotopy formula at the degree $n$), that equals $K_{n+1}$ (as it appears
in the homotopy formula at the degree $n+1$). Finally, if $h>n$, then 
$\tilde K_{h+1} = \delta_c \Delta^{-1}_{\mathbb H,h+1}$ (as it appears
in the homotopy formula at the degree $h$), that equals $K_{h+1}$ (as it appears
in the homotopy formula at the degree $h+1$).

\medskip

Once this point is established, from now on we shall write 
$$
K:= K_{h} =\tilde K_{h}
$$
 without ambiguity. 
 
 \medskip
 
 Therefore $T=\tilde T$ and the homotopy formula \eqref{approx homotopy tilde}
reads as
\begin{equation}\label{approx homotopy}
d_c T+  Td_c + S=I\qquad\mbox{on $B$.} 
\end{equation}
\end{remark}

\medskip

\begin{remark}\label{smoothing negative} By the arguments used in the proof of Theorem \ref{smoothing}, i) the proof
of the $L^p - W^{s,q}$ continuity of $S$ can be adapted to prove that
$S$ is a smoothing operator, i.e for any $m, s\in \mathbb N\cup\{0\}$,
$S$ is bounded from $W^{-m,p}(B',E_0^\bullet)$ to $W^{s,q}(B,E_0^\bullet)$ when \eqref{kpq} holds.
In particular,
 if $\alpha\in W^{-m,p}(B',E_0^\bullet)$
then $S \alpha\in C^\infty(B, E_0^\bullet)$.
\end{remark}

\begin{remark}\label{locality} It is worth pointing out the following fact: take $\alpha, \beta
\in L^p(B',E_0^\bullet)$, $\alpha\equiv \beta$  on $B_1$ ($B_1$ has been introduced in \eqref{varie B}).  Then
$\alpha_0 \equiv \beta_0$  in $B_0$, so that $K_{1,R}\alpha_0 \equiv  K_{1,R} \beta_0$  and
$\tilde K_{1,R} d_c\alpha_0 \equiv  \tilde K_{1,R} d_c \tilde \beta_0$ in $B$. In other words,
$(d_c T +   Td_c)\alpha = (d_c T +   Td_c) \beta$ in $B$. 
Thus, by \eqref{approx homotopy}, $S\alpha = S\beta$ in $B$.

\end{remark}

%

The following commutation lemma will be helpful in the sequel.

\begin{lemma}\label{S-commuta-d} We have:
 $$
 [S, d_c]=0\qquad  \mbox{ in $L^p(\he n, E_0^\bullet)$.}
 $$
\end{lemma}
\begin{proof} Take first $\alpha\in C^\infty(B', E_0^h)$, $1\le h\le 2n+1$. By \eqref{approx homotopy},
$Sd_c=d_cS$ on $ \mc D(B', E_0^h)$.

Take now $\alpha \in L^p (B', E_0^h)$,
and let $\chi_1$ be a cut-off function supported in $B'$, $\chi_1\equiv 1$ on $B_1$
($B_1$ has been defined in \eqref{varie B}). By convolution with usual Friedrichs' mollifiers,
we can find a sequence $(\alpha_k)_{k\in \mathbb N}$ in $\mc D(B', E_0^h)$
converging to $\chi_1\alpha$ in $L^p (B', E_0^h)$. By Theorem \ref{smoothing},
$S\alpha_k \to S(\chi_1\alpha)$ in $W^{2,p}(B,E_0^{h+1})$, and hence $d_c S\alpha_k \to d_cS(\chi_ \alpha)$
in $L^p(B, E_0^h)$ as $k\to\infty$ (obviously, if $h\neq n-1$, it would have been enough
to have $S\alpha_k \to S(\chi_1\alpha)$ in $W^{1,p}(B,E_0^{h+1})$). On the other hand,
$\chi_1\alpha \equiv \alpha$ in $B_1$, and then by Remark \ref{locality} $S(\chi_1\alpha)
= S\alpha$ in $B$, so that $d_c S\alpha_k \to d_cS \alpha$
in $L^p(B, E_0^h)$ as $k\to\infty$.
Moreover
$d_c\alpha_k\to d_c(\chi_1\alpha)$ in $W^{-1,p}(B',E_0^{h})$ (in $W^{-2,p}(B',E_0^{h})$
if $h=n$) and hence, again by Theorem \ref{smoothing}, $Sd_c\alpha_k\to Sd_c(\chi_1\alpha)$
in $B$ as $n\to\infty$. Again $d_c(\chi_1\alpha)\equiv d_c\alpha$ in $B_1$ and then, by
Remark \ref{locality}, $Sd_c\alpha_k\to Sd_c\alpha$
in $B$ as $k\to\infty$. 

Finally, since $d_cS\alpha_k = Sd_c\alpha_k$ for all $k\in \mathbb N$, we can take the
limits as $k\to\infty$ and the assertion follows.
\end{proof}

\begin{theorem}\label{pq poincare}
Let $p,q,h$ s in \eqref{kpq}. With the notations of Definitions \ref{poincare def}, if $1<p<\infty$,  
then both an interior $\he{}$-$\mathrm{Poincar\acute{e}}_{p,q}(h)$ 
and an interior $\he{}$-$\mathrm{Sobolev}_{p,q}(h)$ inequalities hold for $1\le h\le 2n$.
\end{theorem}

\begin{proof} 

\noindent $\he{}$-$\mathrm{Poincar\acute{e}}_{p,q}(h)$ inequality: let $\omega \in L^p(B',E_0^h)$ be such
that $d_c\omega =0$. 

By  \eqref{approx homotopy} we can write $\omega = d_cT\omega + S\omega$ in $B$.  By Remark 
\ref{smoothing negative} and Lemma \ref{S-commuta-d}, $S\omega\in 
\mc E(B,E_0^h)$, and $d_cS\omega =0$. Thus we can apply \eqref{homotopy closed} to $S\omega$
and we get $S\omega = d_cKS\omega$, where $K$ is defined in \eqref{may 4 eq:2}.
In $B$, put now 
$$
\phi:= (KS+T)\omega.
$$ Trivially $d_c\phi = d_cKS\omega + d_cT\omega = S\omega +
\omega - S\omega = \omega$. 
By Theorem \ref{smoothing},
\begin{equation}\label{pq estimates}\begin{split}
\|\phi & \|_{L^q(B,E_0^{h-1})}  \le \|KS\omega\|_{L^q(B,E_0^{h-1})}+ \| T\omega\|_{L^q(B,E_0^{h-1})} 
\\&
\le  \|KS\omega\|_{L^q(B,E_0^{h-1})}+C \|   \omega\|_{L^p(B',E_0^{h-1})}
\\&
\le C\{  \|S\omega\|_{W^{1,q}(B,E_0^{h-1})}+ \|   \omega\|_{L^p(B',E_0^{h-1})}
\} \qquad\mbox{(by Lemma \ref{senza nome})}
\\&
\le C \|   \omega\|_{L^p(B',E_0^{h-1})}.
\end{split}\end{equation}

\noindent $\he{}$-$\mathrm{Sobolev}_{p,q}(h)$ inequality: let $\omega \in L^p(B,E_0^h)$ be 
a compactly supported form such
that $d_c\omega =0$. Since $\omega$ vanishes in a neighborhood of $\partial B$,
without loss of generality we can assume that it is continued by zero on $B'$. In addition,
$\omega=\chi\omega$. By \eqref{approx homotopy}
we have $\omega = d_cT\omega + S\omega$. On the other hand, $T\omega$ is supported
in $B_0$ (since $R$ is small), so that also $S\omega$ is supported
in $B_0$. Again as above $S\omega\in 
C^\infty(B,E_0^h)$, and $d_cS\omega =0$. Thus we can apply \eqref{homotopy closed} to $S\omega$
and we get $S\omega = d_cJS\omega$, where $J$ is defined in \eqref{may 31 eq:2}. By Lemma
\ref{homotopy 1}, $JS\omega$ is supported in $B_0\subset B'$. Thus, if we set
$\phi:= (JS+T)\omega$, then $\phi$ is supported in $B'$. Moreover $d_c\phi = d_cKS\omega + d_cT\omega = S\omega +
\omega - S\omega = \omega$. At this point, we can repeat the estimates \eqref{pq estimates}
and we get eventually
$$
\|\phi  \|_{L^q(B',E_0^{h-1})}  \le C \|   \omega\|_{L^p(B,E_0^{h-1})}.
$$
This completes the proof of the theorem.
\end{proof}

Let $B(p,r)$ a Kor\'anyi ball of center $p\in \he n$ and radius $r>0$. The map $x\to f(x):=  \tau_{p} \delta_{r} (x)$
provides a contact diffeomorphism from $B(e,\rho)$ to $B(p,r\rho)$ for $\rho>0$. Therefore the pull-back $f^\#: E_0^\bullet\to E_0^\bullet$. In addition, if
$\alpha\in E_0^h$, then 
$$
f^\# \alpha = r^h \alpha\circ f\qquad\mbox{if $h\le n$}\qquad\mbox{and}\qquad f^\# \alpha = r^{h+1} \alpha\circ f\qquad\mbox{if $h> n$.}
$$
\begin{theorem} Take $1\le h\le 2n+1$. Suppose  $1<p<Q$ if  $h\not=n+1$ and $1<p<Q/2$ if $h=n+1$. Let $q\ge p$ such that
\begin{eqnarray}\label{kpq 2}
 \frac{1}{p}-\frac{1}{q}\leq \begin{cases}
\frac{1}{Q}      & \text{ if }h\not=n+1, \\
\frac{2}{Q}      & \text{ if }h=n+1.
\end{cases}
\end{eqnarray}
Then
there exists a constant $C$ such that,
 for every $d_c$-closed differential $h$-form $\omega$ in $L^p(B(p,\lambda r);E_0^h)$ there exists a $h-1$-form $\phi$ in $L^q(B(p,r),E_0^{h-1})$ such that $d_c\phi=\omega$ and
\begin{equation*}
\|\phi\|_{L^{q}(B(p,r),E_0^{h-1})}\leq C\, r^{Q/q - Q/p + 1}\,\|\omega\|_{L^{p}(B(p,\lambda r),E_0^h)}\qquad\mbox{if $h\neq n+1$}
\end{equation*}
and
\begin{equation*}
\|\phi\|_{L^{q}(B(p,r),E_0^{n})}\leq C\, r^{Q/q - Q/p + 2}\,\|\omega\|_{L^{p}(B(p,\lambda r),E_0^{n+1}) }.
\end{equation*}
Analogously there exists a constant $C$ such that,
 for every compactly supported $d_c$-closed  $h$-form $\omega$ in $L^p(B(p,r);E_0^h)$ there exists a 
compactly supported $(h-1)$-form $\phi$ in $L^q(B(p,\lambda r),E_0^{h-1})$ such that $d_c\phi=\omega$ in $B(p,\lambda r)$
and
\begin{equation}\label{sobolev interior ball}
\|\phi\|_{L^{q}(B(p,\lambda r),E_0^{k-1})}\leq C\,\|\omega\|_{L^{p}(B(p,r),E_0^k)} \quad 
\end{equation}

\end{theorem}

\begin{proof} We have just to take the pull-back $f^\#\omega$ and then apply Theorem \ref{pq poincare}.

\end{proof}

If the choice of $q$ is sharp (i.e. in \eqref{kpq 2} the equality holds), then the constant on
the right hand side of \eqref{sobolev interior ball} is independent of the radius of the ball,
so that a global $\he{}$-$\mathrm{Sobolev}_{p,q}(h)$ inequality holds.

\begin{corollary}\label{strong sobolev}
Take $1\le h\le 2n+1$. Suppose  $1<p<Q$ if  $h\not=n+1$ and $1<p<Q/2$ if $h=n+1$. Let $q\ge p$ defined by
\begin{eqnarray}\label{kpq 4}
 \frac{1}{p}-\frac{1}{q} := \begin{cases}
\frac{1}{Q}      & \text{ if }h\not=n+1, \\
\frac{2}{Q}      & \text{ if }h=n+1.
\end{cases}
\end{eqnarray}
Then    $\he{}$-$\mathrm{Sobolev}_{p,q}(h)$ inequality holds for $1\le h\le 2n+1$.

\end{corollary}

\section{Contact manifolds and global smoothing}\label{final}

Throughout this section, $(M,H,g)$ will be a sub-Riemannian contact manifold of bounded $C^k$-geometry as in Definition \ref{contact}.
We shall denote by $(E_0^\bullet, d_c)$ both the Rumin's complex in $(M,H,g)$ and in the Heisenberg group.

\begin{proposition}\label{memory} If $\phi $ is a contactomorphism  from an open set $\mc U\subset \he{n}$ to $M$, and we set $\mc V:= \phi(\mc U)$, we have
\begin{itemize}
\item[i)] $\phi^\# E_0^\bullet(\mc V) = E_0^\bullet(\mc U) $;
\item[ii)] $d_c\phi^\# = \phi^\# d_c$;
\item[iii)] if $\zeta$ is a smooth function in $M$, then the differential operator in $\mc U\subset \he n$ defined by $v \to \phi^\#[d_c,\zeta] (\phi^{-1})^\#v$ is a differential operator of order
zero if $v\in E_0^h(\mc U)$, $h\ne n$ and a differential operator of order
1 if $v\in E_0^n(\mc U)$.
\end{itemize}

\end{proposition}

\begin{proof} Assertions i) and ii) follow straightforwardly since $\phi$ is a contact map. Assertion iii) follows from Lemma \ref{leibniz}, since, by
definition, 
$$
 \phi^\#[d_c,\zeta] (\phi^{-1})^\#v = [d_c,\zeta\circ\phi] v. 
$$

\end{proof}

\begin{remark}\label{carte} Let $\{\phi_{x_j}(B(e,1))\}$ a countable locally finite subcovering of $\{\phi_{x}(B(e,1))\, , x\in M\}$. From now on, for sake of
simplicity, we shall write $\phi_j:=\phi_{x_j}$. Without loss of generality, we can replace $B(e,1)$ by $B(e,\lambda)$, where $\lambda >1$ is fixed (just
to be congruent in the sequel with the notations of previous sections).

Let $\{\chi_j\}$ be a partition of the unity subordinated to the covering $\{\phi_j(B(e,\lambda))\}$ of $M$. As above, without loss of
generality, we can assume $\phi_j^{-1}(\mathrm{supp}\;\chi_j) \subset B(e,1)$.
\end{remark}

If $u\in L^p(M, E_0^\bullet)$, we write 
$$
u = \sum_j \chi_j u
$$
We can write 
$$
\chi_j u =  (\phi_j^{-1})^\#\phi_j^\#(\chi_j u) =:   (\phi_j^{-1})^\# v_j.
$$
We use now the homotopy formula in $\he{n}$ (see Theorem \ref{smoothing}):
$$
v_j = d_cT v_j +   Td_cv_j + Sv_j \qquad \mbox{in $B(e,1)$.}
$$
Without loss of generality, we can assume that $R>0$ in the definition of the kernel of $T$  has
been chosen in such a way that the $R$-neighborood of $\phi_j^{-1}(\mathrm{supp}\;\chi_j) \subset  B(e,1)$.
In particular $v_j - d_cT v_j -   Td_cv_j $ is supported in $B(0,1)$ and therefore also $Sv_j$ is supported
in $B(0,1)$.

%
In particular, $ (\phi_j^{-1})^\# \big(
d_cTv_j +    Td_cv_j + Sv_j
\big)$ is supported in $\phi_j (B(e,1))$ so that it can be continued by zero on $M$.

Thus
\begin{equation*}\begin{split}
u &= \sum_j  (\phi_j^{-1})^\# \big(
d_cT v_j +    Td_cv_j + S v_j
\big)
\\&=
d_c \sum_j \ (\phi_j^{-1})^\# 
 T  \phi_j^\#(\chi_j  u )
\\&+
\sum_j ( (\phi_j^{-1})^\# 
  T \phi_j^\#\chi_j ) d_c u 
-
\sum_j  (\phi_j^{-1})^\# 
  T \phi_j^\#([\chi_j ,d_c] u )
\\&+
\sum_j ( (\phi_j^{-1})^\# (S\phi_j^\# \chi_j )u.
\end{split}\end{equation*}
We set
\begin{equation}\label{T}
 Tu:= \sum_j   (\phi_j^{-1})^\# 
 T  \phi_j^\# (\chi_j u)
\end{equation}
and 
\begin{equation}\label{S}
Su:= \sum_j  (\phi_j^{-1})^\# S\phi_j^\# ( \chi_j u) -
\sum_j  (\phi_j^{-1})^\# 
  T \phi_j^\#([\chi_j ,d_c]  u).
\end{equation}

The core of this section consists in the following approximate homotopy formula, where
the ``error term'' $S_M$ has the maximal regularising property compatible with the regularity of $M$.

\begin{theorem}\label{homotopy manifold}
Let $(M,H,g)$ be a bounded $C^k$-geometry sub-Riemannian contact manifold, $k\ge 2$.  Then
\begin{equation}\label{homotopy M}
I= d_c T_M+    T_Md_c + S_M,
\end{equation}
where
$$
T_M:= \big(\sum_{i=0}^{k-1}S^i\big)T,  \qquad S_M:= S^{k},
$$
and $T$ 
and $S$ are
defined in \eqref{T} and \eqref{S}.

By definition
\begin{equation}\label{S commuta su M}
d_c Su = S d_c u.
\end{equation}
In addition, the following maps are continuous:
 \begin{itemize}
 \item[i)] $  T_M: W^{-1,p}(M,E_0^{h+1}) \to L^p(M,E_0^{h})$ if $h\neq n$, and $  T_M: W^{-2,p}(M,E_0^{n+1}) \to L^p(M,E_0^{n})$;
 \item[ii)] $T_M: L^p (M,E_0^{h}) \to W^{1,p}(M,E_0^{h-1})$,   $h\neq n+1$, $ T_M:L^p(M,E_0^{n+1}) \to W^{2,p}(M,E_0^{n})$ if $h=n+1$,
 \item[iii)] $S_M:L^p(M,E_0^h)\to W^{k,p}(M,E_0^{h})$.
\end{itemize}

\end{theorem}

In order to prove Theorem \ref{homotopy manifold}, let us prove the following 
preliminary result:
\begin{lemma}\label{homotopy manifold lemma}
Let $(M,H,g)$ be a bounded $C^k$-geometry sub-Riemannian contact manifold. If $2\le\ell\le k-1$ and $T$
and $S$ are
defined in \eqref{T} and \eqref{S}, then 
\begin{equation}\label{homotopy M0}
I= d_c T+   Td_c + S.
\end{equation}
In addition, the following maps are continuous:
 \begin{itemize}
 \item[i)] $ T: W^{-1,p}(M,E_0^{h+1}) \to L^p(M,E_0^{h})$ if $h\neq n$, and $ T: W^{-2,p}(M,E_0^{n+1}) \to L^p(M,E_0^{n})$;
 \item[ii)] $T: L^p (M,E_0^{h}) \to W^{1,p}(M,E_0^{h-1})$,   $h\neq n+1$, $  T: L^p(M,E_0^{n+1}) \to W^{2,p}(M,E_0^{n})$ if $h=n+1$,
 \item[iii)] if $1\le\ell\le k$, then $S: W^{\ell-1,p}(M, E_0^h) \longrightarrow W^{\ell,p}(M, E_0^h)$.
\end{itemize}

\end{lemma}

\begin{proof} First of all, we notice that, if $\alpha$ is supported in $\phi_j(B(e,\lambda))$, then, by Definition
\ref{contact} the norms
$$
\|\alpha\|_{W^{m,p}(M,E_0^\bullet)} \qquad \mbox{and}\qquad \|\phi_j^\#\alpha\|_{W^{m,p}(\he n,E_0^\bullet)}
$$
are equivalent for $-k\le m\le k$, with equivalence constants independent of $j$. Thus, assertions i) and ii)
follow straightforwardly from Theorem \ref{smoothing}.

\medskip

To get iii)  we only need to note that 
the operators $ (\phi_j^{-1})^\# 
 T \phi_j^\#[\chi_j ,d_c]$ are bounded from $W^{\ell-1,p}(M, E_0^\bullet)\to W^{\ell,p}(M, E_0^\bullet)$ in every degree. 
Indeed, by Lemma \ref{leibniz} above,  the differential operator in $\he n$ $\phi_j^\#[\chi_j ,d_c](\phi_j^{-1})^\#$ has order 1  if $h=n$, and
order 0 if $h\neq n$.  Since the kernel of $ T$ can be estimated by kernel of type 2 if
acts on forms of degree $h=n$, and of type 1 if
acts on forms of degree $h \neq n$, the assertion follows straightforwardly

Summing up in $j$ and keeping into account that the sum is locally finite, we obtain:
\begin{equation*}\begin{split}
\| \sum_j  & \phi_j^\# 
 T_j  (\phi_j^{-1})^\#[\chi_j ,d_c]\|_{W^{\ell,p}(M)}
\le
 \sum_j\| \phi_j^\# 
 T_j  (\phi_j^{-1})^\#[\chi_j ,d_c]\|_{W^{\ell,p}(\phi (\mc U_j))}
\\&
\le
C  \sum_j\| 
 T_j  \phi_j^\#[\chi_j ,d_c]\|_{W^{\ell,p}(\mc U_j) }
\le
C \sum_j 
 \| \phi_j^\# u\|_{W^{\ell-1,p}(U_j)}
\\&
 \le
 C \| u\|_{W^{\ell-1,p}(M)}.
\end{split}\end{equation*}

\end{proof}

\begin{proof}[Proof of Theorem \ref{homotopy manifold}]

By \eqref{S commuta su M}
\begin{equation*}\begin{split}
d_c T_M &+ T_Md_c + S_M 
\\&
= d_c \big(\sum_{i=0}^{k-1}S^i\big)T + \big(\sum_{i=0}^{k-1}S^i\big)\tilde Td_c
+ S^k
\\&
=\sum_{i=0}^{k-1}S^i \big(d_cT +  Td_c\big)
+S^k
\\&
=\sum_{i=0}^{k-1}S^i \big(I-S)
+S^k =I.
\end{split}\end{equation*}
Then statements i), ii) and iii) follow straightforwardly from i), ii) and iii) of
Lemma \ref{homotopy manifold lemma}.

\end{proof}

\section*{Acknowledgments}
B.~F. and A.~B. are supported by the University of Bologna, funds for selected research topics, and by MAnET Marie Curie
Initial Training Network, by 
GNAMPA of INdAM (Istituto Nazionale di Alta Matematica ``F. Severi''), Italy, and by PRIN of the MIUR, Italy.

P.P. is supported by MAnET Marie Curie
Initial Training Network, by Agence Nationale de la Recherche, ANR-10-BLAN 116-01 GGAA and ANR-15-CE40-0018 SRGI. P.P. gratefully acknowledges the hospitality of Isaac Newton Institute, of EPSRC under grant EP/K032208/1, and of Simons Foundation.

\bibliographystyle{amsplain}

\bibliography{BFP2_submitted}

\providecommand{\bysame}{\leavevmode\hbox to3em{\hrulefill}\thinspace}
\providecommand{\MR}{\relax\ifhmode\unskip\space\fi MR }
\providecommand{\MRhref}[2]{%
  \href{http://www.ams.org/mathscinet-getitem?mr=#1}{#2}
}
\providecommand{\href}[2]{#2}
\begin{thebibliography}{10}

\bibitem{BFP}
Annalisa Baldi, Bruno Franchi, and Pierre Pansu, \emph{Gagliardo-{N}irenberg
  inequalities for differential forms in {H}eisenberg groups}, Math. Ann.
  \textbf{365} (2016), no.~3-4, 1633--1667. \MR{3521101}

\bibitem{BFTT}
Annalisa Baldi, Bruno Franchi, Nicoletta Tchou, and Maria~Carla Tesi,
  \emph{Compensated compactness for differential forms in {C}arnot groups and
  applications}, Adv. Math. \textbf{223} (2010), no.~5, 1555--1607.

\bibitem{BFT3}
Annalisa Baldi, Bruno Franchi, and Maria~Carla Tesi, \emph{Hypoellipticity,
  fundamental solution and {L}iouville type theorem for matrix--valued
  differential operators in {C}arnot groups}, J. Eur. Math. Soc. \textbf{11}
  (2009), no.~4, 777--798.

\bibitem{Bernig_2017}
Andreas Bernig, \emph{Natural operations on differential forms on contact
  manifolds}, Differential Geom. Appl. \textbf{50} (2017), 34--51. \MR{3588639}

\bibitem{BEGN}
Robert~L. Bryant, Michael~G. Eastwood, A.~Rod Gover, and Katharina Neusser,
  \emph{{S}ome differential complexes within and beyond parabolic geometry},
  arXiv:1112.2142.

\bibitem{capdangar}
Luca Capogna, Donatella Danielli, and Nicola Garofalo, \emph{Subelliptic
  mollifiers and a basic pointwise estimate of {P}oincar\'e type}, Math. Z.
  \textbf{226} (1997), no.~1, 147--154. \MR{1472145}

\bibitem{folland}
Gerald~B. Folland, \emph{Subelliptic estimates and function spaces on nilpotent
  {L}ie groups}, Ark. Mat. \textbf{13} (1975), no.~2, 161--207. \MR{MR0494315
  (58 \#13215)}

\bibitem{folland_stein}
Gerald~B. Folland and Elias~M. Stein, \emph{Hardy spaces on homogeneous
  groups}, Mathematical Notes, vol.~28, Princeton University Press, Princeton,
  N.J., 1982. \MR{MR657581 (84h:43027)}

\bibitem{FLW_grenoble}
Bruno Franchi, Guozhen Lu, and Richard~L. Wheeden, \emph{Representation
  formulas and weighted {P}oincar\'e inequalities for {H}\"ormander vector
  fields}, Ann. Inst. Fourier (Grenoble) \textbf{45} (1995), no.~2, 577--604.
  \MR{1343563 (96i:46037)}

\bibitem{FSSC_advances}
Bruno Franchi, Raul Serapioni, and Francesco Serra~Cassano, \emph{Regular
  submanifolds, graphs and area formula in {H}eisenberg groups}, Adv. Math.
  \textbf{211} (2007), no.~1, 152--203. \MR{MR2313532 (2008h:49030)}

\bibitem{GromovCC}
Mikhael Gromov, \emph{Carnot-{C}arath\'eodory spaces seen from within},
  Sub-Riemannian geometry, Progr. Math., vol. 144, Birkh\"auser, Basel, 1996,
  pp.~79--323. \MR{MR1421823 (2000f:53034)}

\bibitem{HN}
Bernard Helffer and Jean Nourrigat, \emph{Hypoellipticit\'e maximale pour des
  op\'erateurs polyn\^omes de champs de vecteurs}, Progress in Mathematics,
  vol.~58, Birkh\"auser Boston Inc., Boston, MA, 1985. \MR{MR897103
  (88i:35029)}

\bibitem{IL}
Tadeusz Iwaniec and Adam Lutoborski, \emph{Integral estimates for null
  {L}agrangians}, Arch. Rational Mech. Anal. \textbf{125} (1993), no.~1,
  25--79. \MR{MR1241286 (95c:58054)}

\bibitem{jerison}
David Jerison, \emph{The {P}oincar\'e inequality for vector fields satisfying
  {H}\"ormander's condition}, Duke Math. J. \textbf{53} (1986), no.~2,
  503--523. \MR{MR850547 (87i:35027)}

\bibitem{MSC}
Pierre Maheux and Laurent Saloff-Coste, \emph{Analyse sur les boules d'un
  op\'erateur sous-elliptique}, Math. Ann. \textbf{303} (1995), no.~4,
  713--740. \MR{1359957 (96m:35049)}

\bibitem{mitrea_mitrea_monniaux}
Dorina Mitrea, Marius Mitrea, and Sylvie Monniaux, \emph{The {P}oisson problem
  for the exterior derivative operator with {D}irichlet boundary condition in
  nonsmooth domains}, Commun. Pure Appl. Anal. \textbf{7} (2008), no.~6,
  1295--1333. \MR{2425010}

\bibitem{Pcup}
Pierre Pansu, \emph{Cup-products in $l^{q,p}$-cohomology: discretization and
  quasi-isometry invariance}, arXiv:1702.04984.

\bibitem{rumin_jdg}
Michel Rumin, \emph{Formes diff\'erentielles sur les vari\'et\'es de contact},
  J. Differential Geom. \textbf{39} (1994), no.~2, 281--330. \MR{MR1267892
  (95g:58221)}

\bibitem{Stein}
Elias~M. Stein, \emph{Harmonic analysis: real-variable methods, orthogonality,
  and oscillatory integrals}, Princeton Mathematical Series, vol.~43, Princeton
  University Press, Princeton, NJ, 1993, With the assistance of Timothy S.
  Murphy, Monographs in Harmonic Analysis, III. \MR{MR1232192 (95c:42002)}

\bibitem{VarSalCou}
Nicholas~Th. Varopoulos, Laurent Saloff-Coste, and Thierry Coulhon,
  \emph{Analysis and geometry on groups}, Cambridge Tracts in Mathematics, vol.
  100, Cambridge University Press, Cambridge, 1992. \MR{MR1218884 (95f:43008)}

\end{thebibliography}

\bigskip
\tiny{
\noindent
Annalisa Baldi and Bruno Franchi 
\par\noindent
Universit\`a di Bologna, Dipartimento
di Matematica\par\noindent Piazza di
Porta S.~Donato 5, 40126 Bologna, Italy.
\par\noindent
e-mail:
annalisa.baldi2@unibo.it, 
bruno.franchi@unibo.it.
}

\medskip

\tiny{
\noindent
Pierre Pansu 
\par\noindent Laboratoire de Math\'ematiques d'Orsay,
\par\noindent Universit\'e Paris-Sud, CNRS,
\par\noindent Universit\'e
Paris-Saclay, 91405 Orsay, France.
\par\noindent
e-mail: pierre.pansu@math.u-psud.fr
}

\end{document}